\documentclass[a4paper,11pt]{article}
\usepackage{amsmath}
\usepackage{amssymb}
\usepackage{mathrsfs}
\usepackage{amsfonts}
\usepackage{amsthm}
\usepackage{pstricks, pst-node, pst-text, pst-3d,psfrag}
\usepackage{graphicx}
\usepackage{float}
\usepackage{indentfirst}
\usepackage{enumerate}
\usepackage{subfigure}
\usepackage{tikz}
\usepackage{tikz-cd}
\usepackage{subcaption}
\usepackage{anyfontsize}
\usetikzlibrary{arrows.meta, positioning}
\usepackage{cite}
\usepackage[colorlinks=true]{hyperref}
\hypersetup{urlcolor=blue, citecolor=red}

\usepackage{soul, color, xcolor} 

\theoremstyle{plain}
\newtheorem{theorem}{Theorem}[section]

\newtheorem{proposition}{Proposition}[section]
\newtheorem{corollary}{Corollary}[section]
\newtheorem{lemma}{Lemma}[section]
\theoremstyle{definition}
\newtheorem{definition}{Definition}[section]

\newtheorem{remark}{\textbf{Remark}}[section]

\numberwithin{equation}{section}

\makeatletter 
\@addtoreset{equation}{section}
\makeatother  

\newtheorem{mainthm}{Theorem}[section]

\begin{document}


\title{Differential positivity and dynamical order in\\ noisy  oscillators under unidirectional coupling}

\setlength{\baselineskip}{16pt}

\author{
Bochun Chang, Xiaofang Lin
~and Yi Wang\vspace{0.2cm}\\
School of Mathematical Sciences\\
University of Science and Technology of China\\
Hefei, Anhui, 230026, P. R. China
}


\date{}
\maketitle

\begin{abstract}

We focus on a stochastic system that models a collection of $N$ identical oscillators with unidirectional coupling, perturbed by white noise. The unidirectional coupling breaks the symmetry of mutual dependence among the oscillators. It is shown that the associated random dynamical system $\phi$ admits a dynamical order, meaning that $\phi$ admits a simple asymptotic one-dimensional structure that is totally ordered with respect to the standard order in $\mathbb{R}^N$. Our approach takes a novel geometric perspective: we introduce a random dynamical system $\Phi$ on a smooth Riemannian manifold $M$, diffeomorphic to $\mathbb{S}^1 \times \mathbb{R}^{N-1}$, by ``wrapping'' the random system $\phi$ from $\mathbb{R}^N$ onto $M$. By choosing an appropriate random cone field $\mathcal{C}_M$ on $M$, we show that $\Phi$ is a differentially positive random system on $M$, which is a random counterpart of the differentially positive systems introduced by Forni and Sepulchre. We further demonstrate that the traditional well-known horizontal curves can be identified as the conal curves on $M$, thereby providing a crucial tool for establishing the dynamical order of the random system $\phi$.
\vskip 3mm

\par
\textbf{Keywords}:   Differential positivity,  Dynamical order, Random cone field, Random conal curve, Invariant cone
\end{abstract}

\par \quad \quad \textbf{AMS Subject Classification (2020)}: 37H30, 37C65, 53C30, 60H10, 37D10

\section{Introduction}

In the last four decades, there has been considerable interest in the study of coupled Josephson junctions since it appears in a variety of applied problems including quantum physics, neural networks and chemical reactions (see \cite{ASSBGSD26,CSS10,Kuaramoto26,MOHSBZOS15,MSSPZFH23,SLKSKCS17}). Qian et al. \cite{QianShenZhang90} studied two identical  oscillators coupled with each other under damping and coupling. This model was later generalized in \cite{QianZhuQin97} to a system of $N$ identical overdamped oscillators with nearest-neighbor coupling (cf. Figure \ref{intro-figline-ngbr}). As demonstrated in \cite{QianZhuQin97}, such a system can produce a global attractor that is essentially one-dimensional, indicating that all oscillators will ultimately share the same frequency over an extended period. Subsequently, Qian et al. \cite{QianQinWangZhu00} further studied sine-Gordon equation under Neumann boundary condition (for more details, see \cite{Temam97}), which turned out to be a continuous version of coupled Josephson junctions. They showed that the global attractor is one-dimensional under strong damping and diffusing.
 
    Since thermal fluctuations in Josephson junctions cannot be ignored, a Gaussian white noise term is required to account for their stochastic driving (see, e.g., \cite{ambegaokarHalperin69}). Chow et al. \cite{ChowShenZhou07} first investigated a system of $N$ identically coupled oscillators (as in Figure \ref{intro-figline-ngbr}) subject to white noise. More precisely, they considered the following system of stochastic equations:
    \begin{equation}\label{intro-eqnChow}
        \mathrm{d}\psi=kT\psi \mathrm{d}t+ \begin{pmatrix} \alpha_{1} - \sin\psi_{1} \\ \cdots \\ \alpha_{N} - \sin\psi_{N} \end{pmatrix}\mathrm{d}t+
        \begin{pmatrix}  \epsilon_{1}\mathrm{d}W_1(t) \\ \cdots \\  \epsilon_{N}\mathrm{d}W_N(t) \end{pmatrix},
    \end{equation}
    where $\psi=(\psi_{1},\ldots,\psi_{N})^{\top}$ represents the $N$-angular variables, and $k>0$ is the coupling coefficient. The parameters $\alpha_{i}$ and $ \epsilon_{i}>0$  $(1\le i\le N)$, are input frequencies and  perturbation coefficients, respectively. $T$ is a tridiagonal matrix with main diagonal $(-1,-2,\ldots,-2,-1)$ and off diagonals $(1,\ldots,1)$; and $(W_{i})_{i=1}^N$ is the standard $N$-dimensional Brownian motion. 

     Chow et al. \cite{ChowShenZhou07} have observed and obtained a dynamical order induced by coupling/noise in systems of coupled oscillators. To our knowledge, such a phenomenon can be understood from two perspectives in an integrated manner: on the one hand, it shows that the dynamics of the stochastic system admits an intuitive, simple asymptotic one-dimensional structure, that is, the global random attractor possesses a one-dimensional topological structure; on the other hand, it reveals what can be called a ``total order structure" with respect to the standard order ``$\le$" in $\mathbb{R}^{N}$ induced by the quadrant $\mathbb{R}^{N}_{+}$. As a matter of fact, it deserves to point out that system \eqref{intro-eqnChow} is actually a positive feedback (also called cooperative) system of stochastic equations. Consequently, according to the works of Arnold and Chueshov (see \cite{ArnoldChueshov98,ArnoldChueshov01a,ArnoldChueshov01b,Chueshov02}), such stochastic equations can generate monotone random systems. A fundamental feature of both deterministic and random monotone systems is the so-called exponential separation property (see \cite{LianWang17,PolacikTerescak93,Ruelle79}), governed by the principal one-dimensional direction (also termed principal Floquet spaces in \cite{MierczynskiShen08,MierczynskiShen13a,MierczynskiShen13b,MierczynskiShen16}). From our point of view, this feature provides a theoretical basis for expecting the dynamical order in the positive feedback stochastic system \eqref{intro-eqnChow}. 

    In the present paper, we mainly focus on the stochastic system that models a system of unidirectionally coupled oscillators (cf. Figure \ref{intro-figline-uni})  perturbed by white noise. Compared with nearest-neighbor coupling in Figure \ref{intro-figline-ngbr},  the unidirectional coupling breaks the symmetry of mutual dependence, relying as it does exclusively on the preceding element. In fact, unidirectional coupling inherently falls within the broader framework of nonreciprocal systems which encompasses acoustic, elastic and quantum materials, optoelectronics, photovoltaics (see, e.g., \cite{Nassaretal20,TokuraNagaosa18}). An experimentally accessible instance of such nonreciprocal behavior in a platform is the unidirectionally coupled Josephson junction, as realized in a superconducting device \cite{Ranzanietal17}.
    \begin{figure}[H]
        \begin{minipage}[t]{0.48\textwidth}
            \centering
            \begin{tikzpicture}[circled/.style={draw=blue, circle, minimum size=0.6cm, inner sep=0pt, very thick, text=blue}]
                \node[circled] (1) {\normalsize$1$};
                \node[circled, right=0.8cm of 1] (2) {\normalsize$2$};
                \node[circled, right=0.8cm of 2] (N-1) {\footnotesize$N$-1};
                \node[circled, right=0.8cm of N-1] (N) {\footnotesize$N$};
                \draw[-Latex, orange] ([yshift=2pt]1.east) -- ([yshift=2pt]2.west);
                \draw[-Latex,  orange] ([yshift=-2pt]2.west) -- ([yshift=-2pt]1.east);
                \draw[-Latex, dashed, orange] ([yshift=2pt]2.east) -- ([yshift=2pt]N-1.west);
                \draw[-Latex,dashed,orange] ([yshift=-2pt]N-1.west) -- ([yshift=-2pt]2.east);
                \draw[-Latex,  orange] ([yshift=2pt]N-1.east) -- ([yshift=2pt]N.west);
                \draw[-Latex,orange] ([yshift=-2pt]N.west) -- ([yshift=-2pt]N-1.east);
            \end{tikzpicture}
            \caption{Nearest-neighbor coupling}
            \label{intro-figline-ngbr}
        \end{minipage}
        \hfill
        \begin{minipage}[t]{0.48\textwidth}
            \centering
            \begin{tikzpicture}
            [circled/.style={draw=blue, circle, minimum size=0.6cm, very thick, inner sep=0pt,text=blue}]
                \node[circled] (1) {\normalsize$1$};
                \node[circled, right=0.8cm of 1] (2) {\normalsize$2$};
                \node[circled, right=0.8cm of 2] (N-1) {\footnotesize$N$-$1$};
                \node[circled, right=0.8cm of N-1] (N) {\footnotesize$N$};
                \draw[-Latex, orange] (1.east) -- (2.west);
                \draw[-Latex,  dashed,orange] (2.east) -- (N-1.west);
                \draw[-Latex, orange] (N-1.east) -- (N.west);
            \end{tikzpicture}
            \caption{Unidirectional coupling}
            \label{intro-figline-uni}
        \end{minipage}
    \end{figure}

    We will make an attempt to investigate a dynamical order induced in stochastic system \eqref{intro-eqn-psi} of unidirectionally coupled oscillators. For this purpose, we consider the following  unidirectional coupling system of $N$ identical oscillators driven by additive white noise:
    \begin{equation}\label{intro-eqn-psi}
		\mathrm{d}\psi=kL\psi \mathrm{d}t+ \begin{pmatrix} \alpha_{1}-\sin\psi_{1} \\ \cdots \\ \alpha_{N}-\sin\psi_{N} \end{pmatrix}\mathrm{d}t+
		\begin{pmatrix}  \epsilon_{1}\mathrm{d}W_1(t) \\ \cdots \\  \epsilon_{N}\mathrm{d}W_N(t) \end{pmatrix} ,
	\end{equation}
	where the coupling matrix is

	\begin{equation*}\label{eq-M1}
		L=\begin{pmatrix}
			0 &  & &  \\
			1 & -1 & &  \\
			& \ddots & \ddots &  \\
			& & 1& -1   \\
		\end{pmatrix}.
	\end{equation*}

    Before proceeding further, we first return to the random dynamical system on $\mathbb{R}^N$ generated by system \eqref{intro-eqnChow} with nearest-neighbor coupling. In this context, to obtain the dynamical order, Chow et al. \cite[p.1015]{ChowShenZhou07} cleverly constructed a random family of so-called horizontal curves (first introduced in \cite{QianShenZhang88,QianZhuQin97}). Roughly speaking, a horizontal curve is a non-steep smooth curve in the sense that its Lipschitz constant is uniformly bounded both from above and below. The core idea of their work is to show that each fiber of the global random attractor of the random dynamical system on $\mathbb{R}^N$ generated by system \eqref{intro-eqnChow} is a horizontal curve. A key insight is that their construction of the random family of horizontal curves depends heavily on the bidirectional coupling inherent to stochastic system \eqref{intro-eqnChow}, where each oscillator interacts with both its successor and predecessor (see more details in Remark \ref{do-rmk-difference}).  

    However, for the unidirectional coupling stochastic system \eqref{intro-eqn-psi} in the present paper, there is no such mutual dependence among oscillators: each oscillator is influenced only by its predecessor and receives no influence from its successor. Consequently,  the analytic approach in \cite{ChowShenZhou07,QianShenZhang88,QianZhuQin97} cannot be applied to the corresponding random dynamical system $\phi$ on $\mathbb{R}^N$ generated by  system \eqref{intro-eqn-psi}. This is precisely the situation we encounter when addressing the dynamical order in stochastic system \eqref{intro-eqn-psi}.

    To overcome such difficulties, we move beyond the mindset of conventional analytic estimates and reframe the construction of horizontal curves from a novel geometric viewpoint. More precisely, due to the periodicity of the vector field, we introduce a random dynamical system $\Phi$ (see \eqref{prly-eqn-tildephi}) on a smooth manifold $M$, diffeomorphic to $\mathbb{S}^{1}\times\mathbb{R}^{N-1}$, by ``wrapping"  the random system $\phi$ from $\mathbb{R}^N$ onto $M$. We will show that, by choosing an appropriate random cone field $\mathcal{C}_{M}$ (see \eqref{eq-cone-field}) on $M$, the random system $\Phi$
    turns out to be a differentially positive random system (see Definition \ref{dp-def-dp}), which is a random version of differentially positive systems. Such systems, first introduced by Forni and Sepulchre \cite{Forni15,ForniSepulchre14,ForniSepulchre16}, are nonlinear systems whose linearization along trajectories preserves a cone field on a smooth Riemannian manifold. For a more comprehensive treatment of this topic, we refer the reader to the very recent works in \cite{MostajeranSepulchre18a,MostajeranSepulchre18b,NiuWang25,NiuZhangWang26}. Among the essential concepts of differentially positive systems is the notion of the conal curve, which is a smooth curve whose tangent vector lies in the cone at every point along the curve wherever it is defined (see, e.g.,  \cite[p.5141]{NiuWang25}). More importantly, we point out that the horizontal curves can be identified as the conal curves (see Definition \ref{dp-def-conalcurve}) on $M$, thereby providing a critical tool to investigate dynamical order in stochastic system \eqref{intro-eqn-psi}. Based on such a geometric viewpoint, we present the main result as follows:

    \begin{mainthm}\label{intro-mainthm}
        \textit{The random system $\Phi$ on $M$ possesses a  global random attractor $\mathcal{A}$. Moreover, if the coupling coefficient $k$ in stochastic system \eqref{intro-eqn-psi} is sufficiently large, then the following hold:
        \begin{itemize}
            \item [{\rm(i)}] $\mathcal{A}$ admits a one-dimensional topological structure;  
            \item [{\rm(ii)}] $\mathcal{A}$ is a closed random conal curve on $M$ with respect to $\mathcal{C}_{M}$.
        \end{itemize}}
    \end{mainthm}
    A direct consequence of Theorem \ref{intro-mainthm} is that the unidirectional coupling stochastic system \eqref{intro-eqn-psi} admits \textit{a dynamical order}: $\mathcal{A}$ is not only a one-dimensional global random attractor, but also \textit{the unwrapping of $\mathcal{A}$ from $M$ to $\mathbb{R}^{N}$ is an unbounded smooth curve on each fiber that is totally ordered with respect to the standard order ``$\le$" in $\mathbb{R}^{N}$} (see Corollary \ref{dp-coro-do}).

    This theorem, in a detailed version, will be presented in Section \ref{section-dp} (Theorems \ref{mainth-a} and \ref{mainth-b}), and will be proved in Section \ref{section-pfmainth-b} and Section \ref{section-pfmainth-a}. Our proof of Theorem \ref{intro-mainthm}(i) is motivated by \cite{ChowShenZhou07}, which utilizes the theory of invariant foliation to show that the unwrapping of $\mathcal{A}$ from $M$ to $\mathbb{R}^{N}$ coincides with a one-dimensional invariant manifold of random system $\phi$. Nevertheless, the construction of global random attractor and invariant manifold becomes more delicate in the presence of unidirectional coupling, where the lack of mutual feedback significantly complicates the required a priori estimates. Besides, we further need to prove that $\mathcal{A}$ is a closed curve on each fiber, which turns out to be crucial in proving Theorem \ref{intro-mainthm}(ii). 
    
    More importantly, we emphasize here, to show the total order structure of the unwrapping of $\mathcal{A}$ from $M$ to $\mathbb{R}^{N}$ with respect to ``$\le$" (see Corollary \ref{dp-coro-do}), one needs to construct in the proof of Theorem \ref{intro-mainthm}(ii) an appropriate random cone field $\mathcal{C}_{M}$ on $M$, such that the following two crucial requirements are satisfied:
    \begin{itemize}
        \item [{\rm(a)}]$\Phi$ is differentially positive with respect to $\mathcal{C}_{M}$;
        \item [{\rm(b)}]For any closed random conal curve  on $M$ (with respect to $\mathcal{C}_{M}$), the $\Phi$-forward time evolution admits a subsequence of times along which it converges, in the pull-back sense, to a closed random conal curve.
    \end{itemize}
    Upon finding such a random cone field $\mathcal{C}_{M}$ satisfying (a)-(b), the total order structure (for the unwrapping of $\mathcal{A}$ from $M$ to $\mathbb{R}^{N}$) will be directly deduced by the structures of $\mathcal{A}$ stated in Theorem \ref{intro-mainthm}(i)-(ii) (see Corollary \ref{dp-coro-do}). Unfortunately, the standard cone field $\mathsf{C}_{M}$ induced by $\mathbb{R}^{N}_{+}$ (see \eqref{dp-eqn-quadrantcone}) is not qualified for satisfying the requirement (b), because the random conal curve can be very steep. To overcome such difficulties, we construct a refined proper cone field $\mathcal{C}_{M}\subset \mathsf{C}_{M}$ on $M$ (see \eqref{eq-cone-field}) which not only ensures the differential positivity of the random system $\Phi$ on $M$ (see Proposition \ref{do-prop-cone}), but also simultaneously provides an effective control on the steepness of the evolution of closed random conal curves. This ensures that the global random attractor $\mathcal{A}$ can be approached by iterating such a closed random conal curve with respect to $\mathcal{C}_{M}$ (see Section \ref{section-pfmainth-b}).

    The paper is organized as follows. In Section \ref{section-prly}, we review some notations and basic properties of random dynamical systems. Furthermore, we specify the random system $\Phi$ generated from system \eqref{intro-eqn-psi}. Section \ref{section-dp} introduces the random cone field, random conal curve and differential positivity. Based on this, our main results (Theorems \ref{mainth-a}-\ref{mainth-b} and Corollary \ref{dp-coro-do}) are presented in this section. In Section \ref{section-pfmainth-b}, we focus on the proof of Theorem \ref{mainth-b}. Section \ref{section-pfmainth-a} is devoted to the proof of Theorem \ref{mainth-a} and obtaining the existence of the one-dimensional global random attractor $\mathcal{A}$ via the approaches for constructing the random invariant manifold and its foliation. Finally, in Section \ref{section-simulation}, numerical simulations are presented to illustrate the dynamical order in stochastic system \eqref{intro-eqn-psi}.

    \section{Notations and Preliminaries}\label{section-prly}
    In this section, we present some preliminaries for random dynamical systems and specify how the random system $\Phi$ is generated from the stochastic system \eqref{intro-eqn-psi}. 
    
    We first recall some notations of random dynamical systems (see \cite{Arnold98,Chueshov02,Crauel02} for more details). Let $(M,d_{M})$ be a Polish space equipped with its Borel $\sigma$-algebra $\mathcal{B}(M)$. For $B_{1},B_{2}\subset M$, the \textit{Hausdorff semi-distance} from $B_{1}$ and $B_{2}$ is defined  as  
    \begin{align*}
        \operatorname{dist}(B_1,B_2)=\sup_{x\in B_1}\inf_{y\in B_2}d_{M}(x,y).
    \end{align*} 
    If $B_{1}=\{x\}$, then we denote $d_{M}(x,B_{2})\triangleq\operatorname{dist}(B_{1},B_{2})$. Given a probability space $(\Omega,\mathcal{F},\mathbb{P})$, a multifunction $B\colon \Omega\to2^{M}$ is called a \textit{random set} if $\omega\mapsto d_{M}(x,B(\omega))$ is $(\mathcal{F},\mathcal{B}(\mathbb{R}_{+}))$-measurable, where $\mathcal{B}(\mathbb{R}_{+})$ denotes the Borel $\sigma$-algebra on $\mathbb{R}_{+}$, for all $x\in M$ and $B(\omega)$ is not empty for all $\omega\in\Omega$. If, in addition, each $B(\omega)$ is closed (resp. compact), then $B$ is called a \textit{random closed (resp. compact) set}. A random set $B$ is called \textit{one-dimensional} if $B(\omega)$ is a one-dimensional topological manifold for each $\omega\in\Omega$.
	
	Let  $\theta\colon\mathbb{R}\times\Omega\to\Omega$ be a $(\mathcal{B}(\mathbb{R})\otimes\mathcal{F},\mathcal{F})$-measurable mapping, where $\mathcal{B}(\mathbb{R})$ denotes the Borel $\sigma$-algebra on $\mathbb{R}$. If the family of mappings $(\theta_t)_{t\in\mathbb{R}}$ from $\Omega$ into itself satisfies (1) $\theta_{0}=\operatorname{id}_{\Omega};$ (2) $\theta_{t}\circ\theta_{s}=\theta_{t+s}$ and (3) $\theta_t\mathbb{P}=\mathbb{P}$ for any $t,s\in\mathbb{R}$, then the quadruples $(\Omega,\mathcal{F},\mathbb{P},(\theta_{t})_{t\in\mathbb{R}})$ is called a \textit{metric dynamical system}.

    \begin{definition}
		\rm A mapping $\varphi\colon\mathbb{R}\times M\times\Omega\to M$ defines a \textit{$C^{1}$-random dynamical system} on $M$ over a metric dynamical system $(\Omega,\mathcal{F},\mathbb{P},(\theta_{t})_{t\in\mathbb{R}})$ if
		\begin{itemize}
                \item[{\rm(i)}] $\varphi$ is $(\mathcal{B}(\mathbb{R})\otimes\mathcal{B}(M)\otimes\mathcal{F},\mathcal{B}(M))$-measurable;
                \item[{\rm(ii)}]  $\varphi(t,\cdot,\omega)\colon M\to M$ forms a cocycle over $(\Omega,\mathcal{F},\mathbb{P},(\theta_{t})_{t\in\mathbb{R}})$, that is,
			\begin{align*}
                    &\varphi(0,x,\omega)=x,\quad \omega\in\Omega,\\
				    &\varphi(t+s,x,\omega)=\varphi(t,\varphi(s,x,\omega),\theta_{s}\omega),\quad t,s\in\mathbb{R},\,\omega\in\Omega;
			\end{align*}
            \item[{\rm(iii)}]for each $\omega\in\Omega$, $\varphi(\cdot,\cdot,\omega)\colon\mathbb{R}\times M\to M$ is continuous;
            \item [{\rm (iv)}] for each $(t,\omega)\in\mathbb{R}\times\Omega$, the mapping 
		\begin{align*}
				\varphi(t,\cdot,\omega)\colon M\to M,\quad x\mapsto\varphi(t,x,\omega)
		\end{align*}
            is differentiable with respect to $x$, and the derivative is continuous with respect to $(t,x)$. 
	\end{itemize}   
    \end{definition}

    A random set $B$ is said to be \textit{invariant (resp. positively invariant)} provided that $\varphi(t,B(\omega),\omega)=B(\theta_{t}\omega)$ (resp. $\varphi(t,B(\omega),\omega)\subset B(\theta_{t}\omega)$) for each $t\in\mathbb{R},\omega\in\Omega$. A random set $B$ is called \textit{tempered} if there exist a random variable $r\colon\Omega\to\mathbb{R}_{+}$ and $q\in M$ such that $B(\omega)\subset\{x\in M\mid d_{M}(x,q)\leq r(\omega)\}$ for any $\omega\in\Omega$, where $r$ satisfies $\sup_{t\in\mathbb{R}}e^{-\lambda|t|}r(\theta_{t}\omega)<\infty$ for any $\lambda>0,\omega\in\Omega$. Such an $r$ is called a \textit{tempered random variable}. Let $\mathcal{D}$ be a family of tempered random sets. $B\in\mathcal{D}$ is called \textit{$\mathcal{D}$-absorbing} if, for any $D\in\mathcal{D}$ and $\omega\in\Omega$, there exists $t_{D}(\omega)>0$ such that $\varphi(t,D(\theta_{-t}\omega),\theta_{-t}\omega)\subset B(\omega),t\geq t_{D}(\omega)$. A random set $A$ is said to \textit{attract} another random set $B$ if $\operatorname{dist}\big(\varphi(t,B(\theta_{-t}\omega),\theta_{-t}\omega),A(\omega)\big)\to0$ as $t\to\infty$ for any $\omega\in\Omega$.
	\begin{definition}\label{defi-attractor-compact}\rm
        A random compact set $A\in\mathcal{D}$ is called a \textit{global random attractor} of the random dynamical system $\varphi$ on $M$ if $A$ is invariant and attracts any $D\in\mathcal{D}$.
    \end{definition}
	
    We present a lemma of the existence of global random attractor, whose   measurability can be obtained by the projection theorem in \cite[Theorem III.23]{CastaingValadier77}.
    \begin{lemma}\label{prly-thm-ator}\textup{(\cite[Theorem 3.5]{FlandoliSchmalfuss96})}
        If there exists a $\mathcal{D}$-absorbing random compact set $K$, then for $\omega\in\Omega$,
	\begin{align*}
            \omega\mapsto A(\omega)=\cap_{t>0}\overline{\cup_{\tau\geq t}\varphi(\tau,K(\theta_{-\tau}\omega),\theta_{-\tau}\omega)}
	\end{align*}
        is the unique global random attractor of $\varphi$.
    \end{lemma}

    As a primary example of the abstract framework above, one may consider random ordinary differential equations of the form \begin{equation}\label{prly-MDS-RDE}
	\begin{cases}
		  \dot{x}=f(x,\theta_{t}\omega),\\
	   x(0,\omega)=x_{0},\quad x_{0}\in\mathbb{R}^N,
	\end{cases}
    \end{equation}
    where $f=(f_{1},\ldots,f_{N})\colon\mathbb{R}^{N}\times\Omega\to\mathbb{R}^{N}$ is $(\mathcal{B}(\mathbb{R}^{N})\otimes\mathcal{F},\mathcal{B}(\mathbb{R}^{N}))$-measurable and $\mathcal{B}(\mathbb{R}^{N})$ denotes the Borel $\sigma$-algebra on $\mathbb{R}^{N}$. Denote by $\varphi(t,x_{0},\omega)$ the solution of random system \eqref{prly-MDS-RDE} with $x(0,\omega)=x_{0}$.

    We give the following hypotheses to guarantee that random system \eqref{prly-MDS-RDE} generates a random dynamical system on $\mathbb{R}^{N}$:
    \begin{itemize}
        \item[{\rm\textbf{(H1)}}] There exist random variables $R_{1}$ and $R_{2}$ such that $t\mapsto R_{j}(\theta_{t}\omega)$, $j=1,2$, are locally integrable and 
	\begin{equation*}
            x\cdot f(x,\omega)^{\top}\leq R_{1}(\omega)\|x\|^{2}+R_{2}(\omega),~~ \text{for}~x\in\mathbb{R}^{N}~ \text{and}~\omega\in\Omega.
	\end{equation*}
        \item[{\rm\textbf{(H2)}}] For any $\omega\in\Omega,t\in\mathbb{R},x\in\mathbb{R}^{N}$, $(t,x)\mapsto f(x,\theta_{t}\omega)$ and $(t,x)\mapsto D_{x}f(x,\theta_{t}\omega)$ are continuous, where $D_{x}f(x,\theta_{t}\omega)$ is the Jacobian of $f(\cdot,\theta_{t}\omega)\colon\mathbb{R}^{N}\to\mathbb{R}^{N}$.
        \item[{\rm\textbf{(H3)}}] For each $\omega\in\Omega$ and $x\in\mathbb{R}^{N}$,
	\begin{align*}
            \frac{\partial f_{i}}{\partial{x_{j}}}(x,\omega)\geq0,~~\text{for any}~ i\neq j~\text{and}~1\leq i,j\leq N.
	\end{align*}
    \end{itemize}
	
    \begin{lemma}\label{prly-lem-monotone}
        Let {\rm(H1)-(H2)}hold. Then system \eqref{prly-MDS-RDE} generates a $C^{1}$-random dynamical system $\varphi$ on $\mathbb{R}^{N}$. Moreover, if {\rm(H3)} holds, then for the linearized system of \eqref{prly-MDS-RDE}: 
		\begin{equation*}
			\begin{cases}
    				\dot{\delta x}=D_{x}f(\varphi(t,x_{0},\omega),\theta_{t}\omega)\delta x,\\
				\delta x(0,\omega)=\delta x_{0},\quad \delta x_{0}\in\mathbb{R}^{N},
			\end{cases}
		\end{equation*}
        where $\delta x$ is the variation along trajectories, one has that $\mathbb{R}^{N}_{+}$ and $\operatorname{Int}\mathbb{R}^{N}_{+}$ are positively invariant.
	\end{lemma}
	\begin{proof}
            By similar arguments in \cite[Corollary 2.1.1]{Chueshov02}, (H1) implies that the solution of random system \eqref{prly-MDS-RDE} can be continued to $\mathbb{R}$. Thus, together with \cite[Theorem 2.2.1]{Arnold98}, (H2) implies that \eqref{prly-MDS-RDE} generates a $C^{1}$-random dynamical system $\varphi$. If, in addition, (H3) holds,  similar arguments in \cite[Lemma 5.2.1]{Chueshov02} imply that $\mathbb{R}^{N}_{+}$ and $\operatorname{Int}\mathbb{R}^{N}_{+}$ are positively invariant.  
	\end{proof}
	
        A nonempty closed subset $C\subset \mathbb{R}^{N}$ is called a \textit{closed convex cone} if $C$ satisfies $C+C\subset C$, $\lambda C\subset C$ for $\lambda\geq0$ and $C\cap(-C)=\{0\}$. The cone $C$ naturally induces a (partial) order on $\mathbb{R}^{N}$ by $x_{1}\leq x_{2}$ whenever $x_{2}-x_{1}\in C$. Clearly, $\mathbb{R}_{+}^{N}$ is the  \textit{standard cone} of $\mathbb{R}^{N}$. So, a nonempty set $B\subset\mathbb{R}^{N}$ is said to be \textit{totally ordered} with respect to the standard order ``$\leq$" in $\mathbb{R}^{N}$ if for any two distinct points $x_{1},x_{2}\in B$, $x_{1}-x_{2}\in\mathbb{R}_{+}^{N}\cup (-\mathbb{R}_{+}^{N})$.
	
    Now, we focus on  the stochastic system \eqref{intro-eqn-psi}, which models the $N$ identical  oscillators with unidirectional coupling driven by additive white noise. We show how system \eqref{intro-eqn-psi} can be  converted into an equivalent system of random equations.

    For each one-dimensional Brownian motion $W_{i}$, let 
    \begin{align*}
        \Omega_{0}=C_{0}(\mathbb{R},\mathbb{R})=\{\omega_{0}\mid \omega_{0}\colon\mathbb{R}\to\mathbb{R}~ \text{is~continuous~and~}\omega_{0}(0)=0\}
    \end{align*}
    be equipped with the compact-open topology. Let also $\mathcal{F}_{0}=\mathcal{B}(\Omega_{0})$ be the Borel $\sigma$-algebra of $\Omega_0$, and $\mathbb{P}_{0}$ be the Wiener measure generated by $W_i$. Take
    \begin{align*}
        \bar{\Omega}=\overbrace{\Omega_{0}\times\Omega_{0}\times\ldots\times\Omega_{0}}^{N}~~\text{and}~~ \bar{\mathcal{F}}=\overbrace{\mathcal{F}_{0}\otimes\mathcal{F}_{0}\otimes\ldots\otimes\mathcal{F}_{0}}^{N},
    \end{align*}
    and let $\mathbb{P}$ be the Wiener measure generated by $(W_{i})_{i=1}^N$. Then $(\bar{\Omega},\bar{\mathcal{F}},\mathbb{P})$ is the canonical realization of $(W_{i})_{i=1}^N$. Denote the Wiener shift 
    \begin{align*}
		\theta_{t}\colon\bar{\Omega}\to\bar{\Omega},\quad \theta_{t}\omega(\cdot)=\omega(\cdot+t)-\omega(t).
    \end{align*}
    Then $(\bar{\Omega},\bar{\mathcal{F}},\mathbb{P},(\theta_{t})_{t\in\mathbb{R}})$ is a metric dynamical system.
    
    \begin{lemma}\textup{(\cite[Lemma 2.1]{DuanLuSch03})}\label{pre-lem-hi} 
        Consider Ornstein-Uhlenbeck process 
	\begin{align}\label{eq-OU}
            \mathrm{d}z_{i}+z_{i}\,\mathrm{d}t=\epsilon_i\,\mathrm{d}W_{i}(t),\quad i=1,\ldots,N.
	\end{align}
        There is a $\theta_{t}$-invariant subset $\Omega \subset \bar{\Omega}$ (i.e. $\theta_{t}\Omega = \Omega$ for $t\in\mathbb{R}$) with $\mathbb{P}(\Omega)=1$ such that the following hold for $1\leq i \le N$:
	\begin{enumerate}
            \item [{\rm(i)}] For $\omega \in \Omega$, the random variable 
		\begin{align*}
                \omega\mapsto z_i(\omega)=-\epsilon_{i}\int_{-\infty}^{0}e^{\tau}\omega(\tau)\,\mathrm{d}\tau
		\end{align*}
            exists and generates the unique stationary solution to system \eqref{eq-OU}, which satisfies that $t\mapsto z_{i}(\theta_{t}\omega)$ is continuous in $t\in\mathbb{R}$.
            \item [{\rm(ii)}]$\dfrac{\vert z_{i}({\theta_{t}\omega})\vert}{t}\to0$ as $t\to\pm\infty$ for $\omega\in\Omega$. In particular, $|z_{i}|$ is a 
            tempered random variable.
	\end{enumerate}
    \end{lemma} 
    Let $\Omega$ be as in Lemma \ref{pre-lem-hi} and $\mathcal{F}=\bar{\mathcal{F}}\cap\Omega=\{F\cap\Omega\mid F\in\bar{\mathcal{F}}\}$. Then $(\Omega,\mathcal{F},\mathbb{P},(\theta_t)_{t\in\mathbb{R}})$ is also a metric dynamical system. Denote
    \begin{align}\label{prly-eqn-z}
        z(\theta_{t}\omega)=(z_{1}(\theta_{t}\omega),\ldots,z_{N}(\theta_{t}\omega))^{\top}~~\text{for}~~  t\in\mathbb{R},\,\omega\in\Omega.
    \end{align}
   For the solution $\psi(t,\omega)$ to stochastic systems \eqref{intro-eqn-psi} and $z(\theta_{t}\omega)$ defined in \eqref{prly-eqn-z}, let 
    \begin{align*}
        \phi(t,\omega)=\psi(t,\omega)-z(\theta_{t}\omega)~~\text{for}~~ t\in\mathbb{R},\,\omega\in\Omega.
    \end{align*}
    Then $\phi=(\phi_{1},\cdots,\phi_{N})^{\top}$ satisfies
    \begin{equation}\label{prly-eqn-phi}
        \frac{\mathrm{d}\phi}{\mathrm{d}t}=kL\phi+ \begin{pmatrix} 
            \alpha_1 - \sin(\phi_1 + z_1(\theta_t \omega)) + H_1(\theta_t \omega) \\ \cdots \\ \alpha_N - \sin(\phi_N + z_N(\theta_t \omega)) + H_N(\theta_t \omega)
        \end{pmatrix},
    \end{equation}
    where
    \begin{equation}\label{prly-eqn-H}
        \begin{aligned}
        &H_{1}(\theta_{t}\omega)=z_{1}(\theta_t\omega),~~ \text{and}\\
        & H_{i}(\theta_t\omega)=k[z_{i-1}(\theta_t\omega)-z_i(\theta_t\omega)]+z_i(\theta_t\omega),\quad i=2,\ldots,N.
        \end{aligned}
    \end{equation}
    It then follows from Lemma \ref{prly-lem-monotone} that system \eqref{prly-eqn-phi} generates a $C^{1}$-random dynamical system $\phi$ on $\mathbb{R}^{N}$ over the metric dynamical system $(\Omega, \mathcal{F}, \mathbb{P}, (\theta_{t})_{t \in \mathbb{R}})$; and moreover, $\mathbb{R}^{N}_{+}$ and $\operatorname{Int}\mathbb{R}^{N}_{+}$ are positively invariant with respect to the linearized system $\delta\phi$ (see \eqref{hc-prophc-eqnzeta}). Meanwhile, by the periodicity of random  system \eqref{prly-eqn-phi}, one has 
    \begin{align}\label{prly-eqn-phiperiod}
        \phi(t,\phi_0+(2\pi,\ldots,2\pi)^{\top},\omega)=\phi(t,\phi_0,\omega)+(2\pi,\ldots,2\pi)^{\top}
    \end{align}
    for $t\in\mathbb{R},\omega\in\Omega,\phi_0\in\mathbb{R}^{N}$. Since  $\psi$ and $\phi$ differ only by an explicitly solvable linear correction term,  precisely, \begin{align}\label{prly-eqn-psiphiOU}
     \phi(t,\phi_{0},\omega)=\psi(t,\phi_{0}+z(\omega),\omega)-z(\theta_{t}\omega),\quad \phi_0\in\mathbb{R}^N,\,t\in\mathbb{R},\,\omega\in\Omega,
    \end{align}
    and $\phi$ is a $C^1$-random dynamical system  on $\mathbb{R}^N$ over $(\Omega, \mathcal{F}, \mathbb{P}, (\theta_{t})_{t \in \mathbb{R}})$, we henceforth focus on  $(\Omega, \mathcal{F}, \mathbb{P}, (\theta_{t})_{t \in \mathbb{R}})$ rather than $(\bar{\Omega},\bar{\mathcal{F}},\mathbb{P},(\theta_{t})_{t\in\mathbb{R}})$.
 
    For brevity, we hereafter let the vector $\mathbf{e}=(1,\ldots,1)^{\top}$ and rewrite $2\pi\mathbf{e}$ as $(2\pi,\ldots,2\pi)^{\top}$. By wrapping the phase space $\mathbb{R}^{N}$ along the vector $2\pi\mathbf{e}$, one obtains a manifold $M\triangleq\mathbb{R}^{N}/\mathbb{Z}\cdot2\pi\mathbf{e}$, which is the quotient of $\mathbb{R}^{N}$  by the subgroup $\mathbb{Z}\cdot2\pi\mathbf{e}$ and diffeomorphic to $\mathbb{S}^{1}\times\mathbb{R}^{N-1}$. The metric $d_{\mathbb{R}^{N}}$ on $\mathbb{R}^{N}$, defined by the standard Euclidean norm $\|\cdot\|$,   naturally induces a metric $d_{M}$ on $M$. More precisely, for any $x_{i}\in\mathbb{R}^{N}$, $i=1,2$, let 
    \begin{align*}
       \hat{x}_{i}\triangleq x_{i}\,\bmod\,2\pi\mathbf{e}\in M,\quad i=1,2.
    \end{align*}
    The distance between $\hat{x}_{1}$ and $\hat{x}_{2}$ on $M$ is
    \begin{align}\label{prly-eqn-dM}
        d_{M}(\hat{x}_{1},\hat{x}_{2})\triangleq\inf_{n\in\mathbb{Z}}d_{\mathbb{R}^{N}}\big(x_{1},x_{2}+n\cdot2\pi\mathbf{e}\big)=\inf_{n\in\mathbb{Z}}
        \|x_{1}-x_{2}-n\cdot2\pi\mathbf{e}\|.
    \end{align}
    Together with \eqref{prly-eqn-phiperiod}, the system $\phi$ can induce a \textit{new random dynamical system} $\Phi$ on $M$ as follows
    \begin{align}\label{prly-eqn-tildephi}
        \Phi(t,\Phi_{0},\omega)\triangleq\phi(t,\phi_{0},\omega)\,\bmod\,2\pi\mathbf{e},
    \end{align}
    where $\Phi_{0}\triangleq\phi_{0}\,\bmod\,2\pi\mathbf{e}\in M$ and $\phi^{0}\in\mathbb{R}^{N}$. We will focus on the random system $\Phi$ in Section \ref{section-dp} to prepare for the dynamical order in stochastic system \eqref{intro-eqn-psi}.

	\section{Differential Positivity of \texorpdfstring{$\Phi$}{tilde{phi}} and Main Results}\label{section-dp}
     
    In this section, we will first show the differential positivity of random system $\Phi$ on $M$, and then we will present the main results in our paper (Theorems \ref{mainth-a}-\ref{mainth-b} and Corollary \ref{dp-coro-do}). Forni and Sepulchre \cite{Forni15,ForniSepulchre14,ForniSepulchre16} first introduced the differential positivity for deterministic systems, by which they investigated the dynamics of nonlinear pendulums, nonlinear consensus protocols and decision-making processes, etc. For more recent works, we refer to \cite{MostajeranSepulchre18a,MostajeranSepulchre18b,NiuWang25,NiuZhangWang26}. Motivated by their works, we here introduce the differential positivity for random dynamical systems. Before giving the rigorous definition, we first introduce a random cone field on $M$ (see \cite{ForniSepulchre16,Lawson89,NiuWang25} for deterministic version). 
    \subsection{Differential Positivity of \texorpdfstring{$\Phi$}{tilde{phi}}}

    Let $M$ be a smooth manifold of dimension $N$, endowed with a Riemannian metric tensor and the Riemannian metric $d_{M}$. We assume that $(M,d_{M})$ is a complete metric space. The tangent bundle is denoted by $TM$ and the tangent space at a point $y\in M$ by $T_{y}M$. 
    \begin{definition}\label{de-cf}
        \rm A \textit{random cone field} on a manifold $M$ is a map $C_{M}\colon(\omega,y)\mapsto C^{\omega}_{M}(y)$ such that $C^{\omega}_{M}(y)$ is a closed convex cone in $T_{y}M$ for each $y\in M,\omega\in\Omega$. In local coordinates, $C_{M}^{\omega}(y)$ can be represented as  
        \begin{align*}
            C_{M}^{\omega}(y)=\{\delta y\in T_{y}M\mid g_{i}(y,\delta y,\omega)\geq0,i\in I\},
        \end{align*}
        where $\delta y$ is a variation (tangent vector) at $y$, $I\subset\mathbb{Z}$ is an index set and $g_{i}\colon TM\times\Omega\to\mathbb{R},i\in I,$ are functions such that 
        \begin{itemize}
            \item [{\rm(i)}] $g_{i}(\cdot,\cdot, \omega)$ are continuous for each $\omega\in\Omega$;
            \item [{\rm(ii)}]$g_{i}(y,\delta y,\cdot)$ are $(\mathcal{F},\mathcal{B}(\mathbb{R}))$-measurable for each $(y,\delta y)\in TM$.
        \end{itemize}
    \end{definition}
    \vspace{0.25\baselineskip} 
    
    With the random cone field $C_{M}$, one can define the differentially positive random system as one whose linearization along trajectories preserves $C_{M}$.
    \begin{definition}\label{dp-def-dp}
        \rm A $C^{1}$-random dynamical system $\varphi$ on $M$ is said to be \textit{differentially positive} with respect to a random cone field $C_{M}$ if
	\begin{align*}
            \mathrm{d}\varphi(t,y,\omega)C^{\omega}_{M}(y)\subset C_{M}^{\theta_{t}\omega}(\varphi(t,y,\omega)),\quad \forall~t\geq0,y\in M,\,\omega\in\Omega,
	\end{align*}
        where $\mathrm{d}\varphi(t,y,\omega)$ is the tangent map from $T_{y}M$ to $T_{\varphi(t,y,\omega)}M$.
    \end{definition}

    For the random system $\Phi$ on the manifold $M$ defined in \eqref{prly-eqn-tildephi}, due to the periodicity and the positive invariance of $\mathbb{R}_{+}^{N}$  for the corresponding linearized system $\delta\phi$ of system \eqref{prly-eqn-phi} (see \eqref{hc-prophc-eqnzeta}), a standard random cone field that makes random system $\Phi$ differentially positive is 
         \begin{align}\label{dp-eqn-quadrantcone}
             \mathsf{C}_{M}\colon(\omega,y)\mapsto\mathsf{C}_{M}^{\omega}(y)\triangleq\{\delta y\in T_{y}M\mid \delta y\in  \mathbb{R}^{N}_{+}\}.
         \end{align}
         However, such a cone field $\mathsf{C}_{M}$ is not qualified for satisfying the requirement (b) mentioned in the Introduction. Consequently, we have to construct a new refined proper cone field $\mathcal{C}_{M}\subset\mathsf{C}_{M}$ on $M$ as follows:
         \begin{align}\label{eq-cone-field}
            (\omega,y)\mapsto \mathcal{C}_{M}^{\omega}(y)\triangleq\{\delta y\in T_{y}M\mid g_{j}(y,\delta y,\omega)\geq0,~ j=1,\ldots,2N-2\},
	\end{align}
	where
        \begin{align*}
            g_{2i-1}(y,\delta y,\omega)=\delta y_{i+1}-\gamma_{i}{\delta y}_{i}~~\text{with}~~ \gamma_{i}^{-1}=1+2ik^{-1}, 
        \end{align*}
        and
	\begin{align*}
		g_{2i}(y,\delta y,\omega)=\beta_{i}{\delta y}_{i}-{\delta y}_{i+1}~~\text{with}~~\beta_{i}^{-1}=1-2ik^{-1},
	\end{align*}
	for $i=1,\ldots,N-1$. Here $k$ is the coupling coefficient, and ${\delta y}_{i}$ is the $i$-th component of $\delta y$. The following proposition indicates that random system $\Phi$ is differentially positive with respect to $\mathcal{C}_{M}$.

	\begin{proposition}\label{do-prop-cone}
		The random system $\Phi$ on $M$ is differentially positive with respect to $\mathcal{C}_{M}$ defined in \eqref{eq-cone-field}.
    \end{proposition}
    \begin{proof}
        By the definition of random system $\Phi$ on  $M$ (see \eqref{prly-eqn-tildephi}), the differential positivity of random system $\Phi$ is equivalent to the corresponding linearized system along trajectories of random system $\phi$ on $\mathbb{R}^{N}$ preserving the polyhedral cone
        \begin{align}\label{dp-propdp-eqnconeC}
            \mathscr{C}=\{\delta x\in\mathbb{R}^{N}_{+}\mid\gamma_{i}{\delta x}_{i}\leq {\delta x}_{i+1}\leq\beta_{i}{\delta x}_{i},i=1,\ldots,N-1\},
	\end{align}
        where ${\delta x}_{i}$ denotes the $i$-th component of $\delta x$ (see Figure \ref{dp-propdpfig-coneC}). Note that the corresponding linearized system $\delta\phi$ along trajectories of random system $\phi$ on $\mathbb{R}^{N}$ is	\begin{equation}\label{hc-prophc-eqnzeta}
		\begin{cases}
                \dfrac{\mathrm{d}\delta\phi}{\mathrm{d}t}=kL \delta\phi+ \begin{pmatrix} -\cos\left(\phi_{1}+z_{1}\right)&& \\ 
	       	&\ddots&\\ 
                &&-\cos\left(\phi_{N}+z_{N}\right) \end{pmatrix}\delta\phi,\\
				\delta\phi(0,\omega)=\delta\phi_{0}.
		\end{cases}
	\end{equation}
        Denote by $\delta\phi(t,\delta\phi_{0},\omega)$ the solution of random system \eqref{hc-prophc-eqnzeta}. It suffices to verify $\delta\phi(t,\delta\phi_{0},\omega)\in\mathscr{C}$, whenever $\delta\phi_{0}\in\mathscr{C}$ and $t\geq0$. For this purpose, let $\mathscr{C}_{i}$ be projection regions of $\mathscr{C}$ in $(\delta x_{i},\delta x_{i+1})$-plane, that is, 
         \begin{align*}
            \mathscr{C}_i=\{({\delta x}_{i},{\delta x}_{i+1})^{\top}\in\mathbb{R}^{2}_{+}\mid\gamma_i {\delta x}_i \leq {\delta x}_{i+1} \leq \beta_{i}{\delta x}_{i}\},\quad i=1,\ldots,N-1.
        \end{align*} 
        Then, one needs to show that on each $\partial\mathscr{C}_i$ (the boundary of $\mathscr{C}_{i}$), the projected direction of $\delta\phi(t,\delta\phi_{0},\omega)$ is not outward from $\mathscr{C}_{i}$ for $i=1,\ldots,N-1$.
                       \begin{figure}[H]
            \begin{minipage}[t]{0.44\linewidth}
                \centering
                \includegraphics[width=1.0\textwidth]{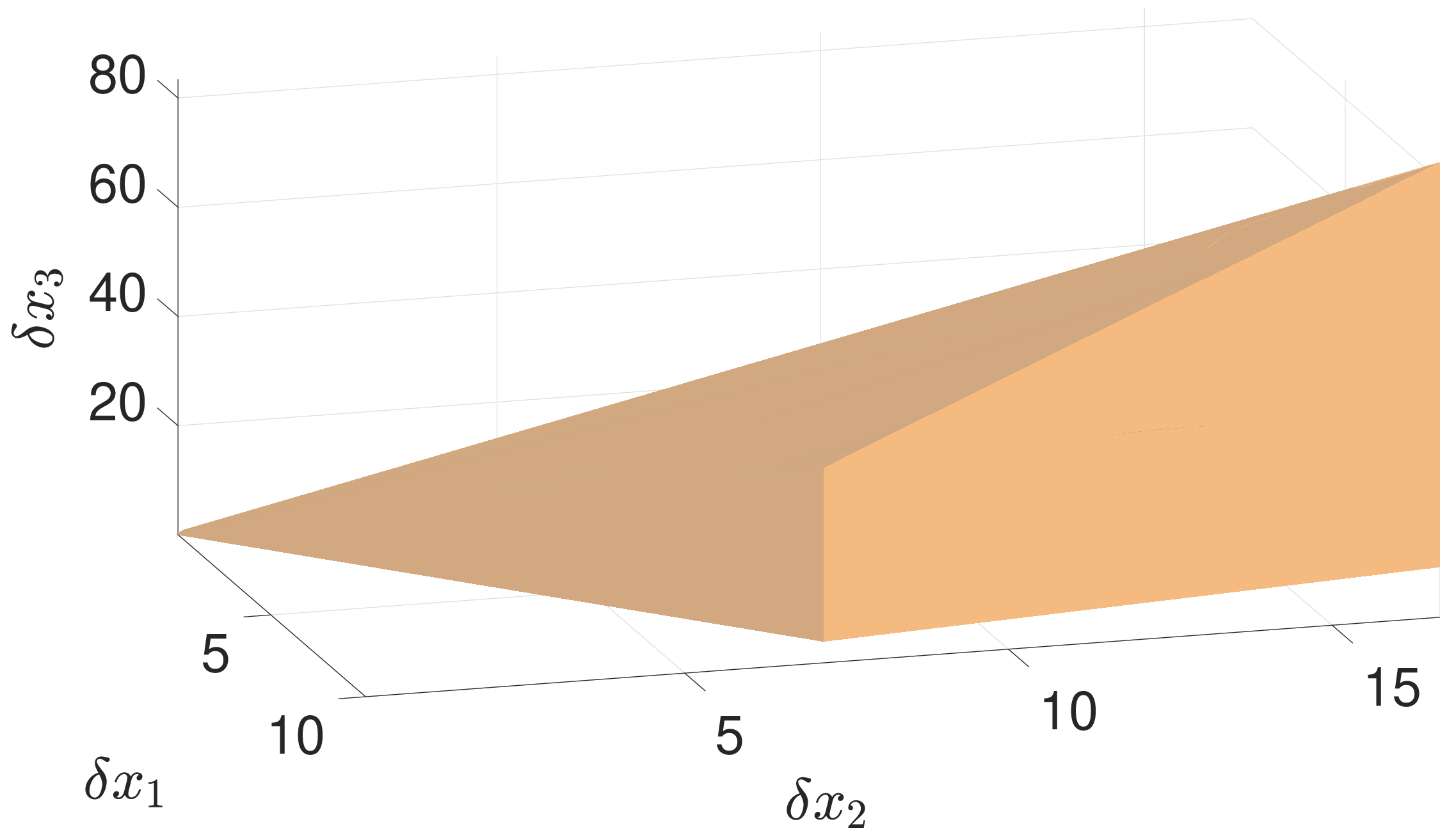}
                \caption{ The polyhedral cone $\mathscr{C}$ with $N=3$ and $k=5$}
                \label{dp-propdpfig-coneC}
            \end{minipage}
            \hfill
            \begin{minipage}[t]{0.44\linewidth}
                \centering
                \begin{tikzpicture}[scale=0.6] 
                    \coordinate (A) at (0,0);
                    \coordinate (B) at (3.5,0);
                    \coordinate (C) at (3,1);
                    \coordinate (D) at (1.5,3);
                    \coordinate (E) at (0,3.5);
                    \coordinate (F) at (-1.2,0.4);
                    \coordinate (G) at (0.95,1.6);
                    \coordinate (H) at (0.4,-1.2);
                    \coordinate (R) at (1.3,0.8);
                    \fill[gray!30, opacity=1] (A) -- (C) -- (D) -- cycle;
                    \draw[->]  (A) to (B);
                    \draw[-]   (A) to (C);
                    \draw[-]   (A) to (D);
                    \draw[->]  (A) to (E);
                    \draw[blue,->]  (A) to (F);
                    \draw[blue,->]  (A) to (G);
                    \draw[blue,->]  (A) to (H);
                    \draw (E) node[above] {${\delta x}_{i+1}$};
                    \draw (D) node[right] {\footnotesize${\delta x}_{i+1}=\beta_{i}{\delta x}_{i}$};
                    \draw (C) node[right] {\footnotesize${\delta x}_{i+1}=\gamma_{i}{\delta x}_{i}$};
                    \draw (B) node[right] { ${\delta x}_{i}$};
                    \draw (G) node[right] {\tiny \hspace{-1.2mm}$\left(\!\frac{\mathrm{d}{\delta\phi}_{i}}{\mathrm{d}t}\!(t_{0},\!\delta\phi_{0},\!\omega),\!\!\frac{\mathrm{d}{\delta\phi}_{i+1}}{\mathrm{d}t}\!(t_{0},\!\delta\phi_{0},\!\omega)\!\right)^{\top}$};
                    \draw (F) node[above right] {\footnotesize$n_{\beta_{i}}$};
                    \draw (H) node[right] {\footnotesize$n_{\gamma_{i}}$};
                    \draw (R) node[right] {\footnotesize$\mathscr{C}_{i}$};
                \end{tikzpicture}
                \caption{Vector fields on $\mathscr{C}_{i}$ and unit outward normal vectors of $\partial\mathscr{C}_{i}$}
                \label{dp-propdpfig-coneC_i}
		  \end{minipage}
	\end{figure}

         To this end, we fix $\omega\in \Omega$ and $i$. Note that, in $(\delta x_{i},\delta x_{i+1})$-plane, the unit outward normal vectors of $\partial\mathscr{C}_{i}$ are 
	\begin{align*}
            n_{\beta_{i}}=\left(-\frac{1}{\sqrt{1+\frac{1}{\beta_i^{2}}}},\frac{1}{\beta_{i}\sqrt{1+\frac{1}{\beta_i^{2}}}}\right)^{\top} ~~\text{and}~~ n_{\gamma_{i}}=\left(\frac{1}{\sqrt{1+\frac{1}{\gamma_i^{2}}}},-\frac{1}{\gamma_{i}\sqrt{1+\frac{1}{\gamma_i^{2}}}}\right)^{\top},
	\end{align*}
        respectively. We will show that, if 
        $\delta\phi(t_{0},\delta\phi_{0},\omega)\in\mathscr{C}$, i.e.,
        \begin{align}\label{dp-propdp-eqndeltaphi}
            \gamma_j{\delta\phi}_{j}(t_{0},\delta\phi_{0},\omega)\leq{\delta\phi}_{j+1}(t_{0},\delta\phi_{0},\omega)\leq \beta_{j}{\delta\phi}_{j}(t_{0},\delta\phi_{0},\omega),~~\text{for any}~~ j\neq i,
        \end{align}
        for some $(t_{0},\delta\phi_{0})\in [0,\infty)\times\mathscr{C}$ with ${\delta\phi}_{i+1}(t_{0},\delta\phi_{0},\omega)=\beta_{i}{\delta\phi}_i(t_{0},\delta\phi_{0},\omega)$, then it must hold that
	\begin{equation}\label{dp-propdp-eqnbeta}
            \left(\frac{\mathrm{d}{\delta\phi}_{i}}{\mathrm{d}t}(t_{0},\delta\phi_{0},\omega),\frac{\mathrm{d}{\delta\phi}_{i+1}}{\mathrm{d}t}(t_{0},\delta\phi_{0},\omega)\right)^{\top}\cdot n_{\beta_{i}}\leq0
	\end{equation}
        (see Figure \ref{dp-propdpfig-coneC_i}). Hereafter, we omit $(t_{0},\delta\phi_{0},\omega)$ for simplicity. In order to show \eqref{dp-propdp-eqnbeta}, for $i=1$, \eqref{dp-propdp-eqnbeta} is equivalent to
	\begin{equation}\label{hc-prophc-eq5}
            \begin{aligned}
                \cos\left(\phi_{1}+z_{1}\right)\cdot{\delta\phi}_{1}+\frac{1}{\beta_{1}}\Big[-\cos\left(\phi_{2}+z_{2}\right)\cdot{\delta\phi}_{2}+k\left({\delta\phi}_{1}-{\delta\phi}_{2}\right)\Big]\leq 0.
            \end{aligned}
        \end{equation}
        Note that ${\delta\phi}_2=\beta_{1}{\delta\phi}_1$. Then \eqref{hc-prophc-eq5} can be rewritten as
        \begin{align*}
            \cos\left(\phi_{1}+z_{1}\right)\cdot{\delta\phi}_{1}-\cos\left(\phi_{2}+z_{2}\right)\cdot{\delta\phi}_{1}+\frac{k}{\beta_{1}}{\delta\phi}_{1}-k{\delta\phi}_{1}\leq \left(2+\frac{k}{\beta_{1}}-k\right){\delta\phi}_{1}.
        \end{align*}
        Since $\frac{1}{\beta_1}=1-\frac{2}{k}$, we have $\left(2+\frac{k}{\beta_1}-k\right){\delta\phi}_1 \leq 0$, whence \eqref{dp-propdp-eqnbeta} holds for $i=1$. While, for $2\leq i\leq N-1$, \eqref{dp-propdp-eqnbeta} is equivalent to
        \begin{equation}\label{hc-prophc-eq7}
            \begin{aligned}
                \cos\left(\phi_{i}+z_{i}\right)\cdot{\delta\phi}_{i}-
                &k\left({\delta\phi}_{i-1}-{\delta\phi}_{i}\right)+\\
                &\frac{1}{\beta_{i}}\Big[-\cos\left(\phi_{i+1}+z_{i+1}\right)\cdot{\delta\phi}_{i+1}+k\left({\delta\phi}_{i}-{\delta\phi}_{i+1}\right)\Big]\leq 0.
            \end{aligned}
        \end{equation}
        Again, noticing that ${\delta\phi}_{i+1}=\beta_{i}{\delta\phi}_{i}$, one can simplify \eqref{hc-prophc-eq7} as
        \begin{equation}\label{hc-prophc-eq8}
            \begin{aligned}
            \cos\left(\phi_{i}+z_{i}\right)\cdot{\delta\phi}_{i}-k{\delta\phi}_{i-1}-\cos\left(\phi_{i+1}+z_{i+1}\right)\cdot{\delta\phi}_{i}+\frac{k}{\beta_{i}}{\delta\phi}_{i}\leq 0. 
            \end{aligned}
        \end{equation}
        Recall that ${\delta\phi}_{i}\leq\beta_{i-1}{\delta\phi}_{i-1}$ since $\delta\phi\in\mathscr{C}$ (see \eqref{dp-propdp-eqndeltaphi}). Then \eqref{hc-prophc-eq8} can be rewritten as
    {\allowdisplaybreaks\begin{align*}
            &\cos\left(\phi_{i}+z_{i}\right)\cdot{\delta\phi}_{i}-k{\delta\phi}_{i-1}-\cos\left(\phi_{i+1}+z_{i+1}\right)\cdot{\delta\phi}_{i}+\frac{k}{\beta_{i}}{\delta\phi}_{i} \\
            \leq&\cos\left(\phi_{i}+z_{i}\right)\cdot{\delta\phi}_{i}-\frac{k}{\beta_{i-1}}{\delta\phi}_{i}-\cos\left(\phi_{i+1}+z_{i+1}\right)\cdot{\delta\phi}_{i}+\frac{k}{\beta_{i}}{\delta\phi}_{i}\\
            \leq&\left(2-\frac{k}{\beta_{i-1}}+\frac{k}{\beta_{i}}\right){\delta\phi}_{i}.
        \end{align*}}With $\frac{1}{\beta_i}=1-\frac{2}{k}i$ for $1\leq i \leq N -1$, one obtains that $\left(2-\frac{k}{\beta_{i-1}}+\frac{k}{\beta_{i}}\right){\delta\phi}_{i}\leq0$, which confirms \eqref{hc-prophc-eq7} (and hence, \eqref{dp-propdp-eqnbeta} holds) for $2\leq i\leq N-1$. Thus, we have proved \eqref{dp-propdp-eqnbeta}.

        Similarly, by repeating the same arguments, one can obtain that
        \begin{equation}\label{dp-propdp-eqngamma}
            \left(\frac{\mathrm{d}{\delta\phi}_{i}}{\mathrm{d}t}(t_{0},\delta\phi_{0},\omega),\frac{\mathrm{d}{\delta\phi}_{i+1}}{\mathrm{d}t}(t_{0},\delta\phi_{0},\omega)\right)^{\top}\cdot n_{\gamma_{i}}\leq0,	
	\end{equation}
        whenever $\delta\phi(t_{0},\delta\phi_{0},\omega)\in\mathscr{C}$ and ${\delta\phi}_{i+1}(t_{0},\delta\phi_{0},\omega)=\gamma_{i}{\delta\phi}_i(t_{0},\delta\phi_{0},\omega)$.

        Together with \eqref{dp-propdp-eqnbeta} and \eqref{dp-propdp-eqngamma}, we have obtained that, on each $\partial\mathscr{C}_i$, the projected direction of $\delta\phi(t,\delta\phi_{0},\omega)$ is not outward from $\mathscr{C}_{i}$ for $1\leq i\leq N-1$. Thus, we have obtained that $\delta\phi(t,\delta\phi_{0},\omega)\in\mathscr{C}$ for any $\delta\phi_{0}\in\mathscr{C}$, which completes the proof that $\Phi$ on $M$ is differentially positive with respect to $\mathcal{C}_{M}$.
    \end{proof}

    Now, in order to state our main results, we introduce the random conal curve with respect to $C_{M}$. For deterministic counterpart, one may refer to \cite{ForniSepulchre16,Lawson89,NiuWang25}.
  
    \begin{definition}\label{dp-def-conalcurve}
        \rm Let $\varphi$ be a random system on $M$ equipped with a random cone field $C_{M}$. Let also $[s_{0},s_{1}]$ be an interval (not necessarily bounded) in $\mathbb{R}$. A map $\eta\colon[s_{0},s_{1}]\times\Omega\to M;(s,\omega)\mapsto\eta(s,\omega)$ is called a \textit{random conal curve} with respect to $C_{M}$ if the following hold:
        \begin{itemize}
            \item [{\rm(i)}] $\eta(\cdot,\omega)\colon[s_{0},s_{1}]\to M$ is continuously differentiable for each $\omega\in\Omega$;
            \item [{\rm(ii)}] $\eta'(s,\omega)\triangleq\frac{\mathrm{d}}{\mathrm{d}s}\eta(s,\omega)\in C_{M}^{\omega}(\eta(s,\omega))$ for any $s\in [s_{0},s_{1}],\omega\in\Omega$;
            \item [{\rm(iii)}] $\eta([s_0,s_1],\cdot)\colon\omega\mapsto\eta([s_0,s_1],\omega)\triangleq\{\eta(s,\omega),s\in[s_{0},s_{1}]\}$ is a random set on $M$.
        \end{itemize}
    \end{definition}
    Moreover, a random conal curve $\eta$ is \textit{closed} if $\eta(s_{0},\omega)=\eta(s_{1},\omega)$  for each $\omega\in\Omega$; $\eta$ is \textit{tempered} if $\omega\mapsto\eta([s_{0},s_{1}],\omega)$ is a tempered random set; and $\eta$ is \textit{invariant} if $\varphi(t,\eta([s_{0},s_{1}],\omega),\omega)=\eta([s_{0},s_{1}],\theta_{t}\omega)$ for any $t\in\mathbb{R},\omega\in\Omega$.\vspace{0.25\baselineskip} 
    
    A useful property of the random conal curve is the following
    \begin{proposition}\label{dp-prop-conalcurve}
         Let $\varphi$ be a differentially positive random system on $M$ with respect to  $C_{M}$. Let also $\eta\colon[s_{0},s_{1}]\times\Omega\to M$ be a random conal curve. Then, for each $\omega\in\Omega$ and $t\geq0$, the $\varphi$-forward time evolution  $\varphi(t,\eta(\cdot,\omega),\omega)$ is a random conal curve. 
    \end{proposition}
    \begin{proof}
       Fix $t\geq0$ and $\omega\in\Omega$. Clearly,  the continuous differentiability of $\varphi(t,\eta(\cdot,\allowbreak \omega), \omega)\colon [s_{0},s_{1}]\to M$ follows from those of $\varphi(t,\cdot,\omega)$ and $\eta(\cdot,\omega)$. In addition,    
        \begin{align*}
            \frac{\mathrm{d}}{\mathrm{d}s}\varphi(t,\eta(s,\omega),\omega)&=\mathrm{d}\varphi(t,\eta(s,\omega),\omega)\frac{\mathrm{d}}{\mathrm{d}s}\eta(s,\omega)\\
            &\in\mathrm{d}\varphi(t,\eta(s,\omega),\omega)C_{M}^{\omega}(\eta(s,\omega))\subset C_{M}^{\theta_{t}\omega}(\varphi(t,\eta(s,\omega),\omega)),
        \end{align*}
       for any $s\in[s_0,s_1]$. This is due to the fact that $\eta$ is a random conal curve and $\varphi$ is differentially positive with respect to $C_{M}$. Moreover, recall that $\eta([s_{0},s_{1}],\cdot)$ is a random set. It then follows from the continuity of $\varphi(t,\cdot,\omega)$ and $(\mathcal{B}(M)\otimes\mathcal{F},\mathcal{B}(M))$-measurability of $\varphi(t,\cdot,\cdot)$ that the $\varphi$-forward time evolution of $\eta$ is a random conal curve with respect to $C_{M}$.
    \end{proof}

    \subsection{Main Results} Now, we are ready to present our main results in this paper.
    
    \begin{theorem}\label{mainth-a}
        For the random system $\Phi$ in \eqref{prly-eqn-tildephi} on $M$, there exists a global random attractor $\mathcal{A}\colon\Omega\to2^{M}$ such that, for any $M$-valued tempered random set $D$ and $\omega\in\Omega$, 
        \begin{align}\label{dp-maintha-attract}
            \operatorname{dist}\big(\Phi(t,D(\theta_{-t}\omega),\theta_{-t}\omega),\mathcal{A}(\omega)\big)\to0,~~\text{as}~~ t\to\infty.
        \end{align}
        Moreover, if coupling coefficient $k$ in \eqref{intro-eqn-psi} is sufficiently large, then  for each $\omega\in \Omega$, $\mathcal{A}(\omega)$ is a one-dimensional  $C^{1}$-smooth closed curve. 
    \end{theorem}
    
        With the help of the differential positivity of random system $\Phi$ with respect to $\mathcal{C}_{M}$ on $M$ (see Proposition \ref{do-prop-cone}), we  further obtain that $\mathcal{A}$ is a random conal curve with respect to $\mathcal{C}_{M}$:
    \begin{theorem}\label{mainth-b}
        Let all hypotheses hold in Theorem \ref{mainth-a}. Then, the global random attractor $\mathcal{A}$ is a closed random conal curve, denoted by $l\colon[0,2\pi]\times\Omega\to M$, with respect to $\mathcal{C}_{M}$ in \eqref{eq-cone-field}. More precisely, for $\omega\in\Omega$, one has $\mathcal{A}(\omega)=\left\{l(\tau,\omega)\mid \tau\in[0,2\pi]\right\}$,
        \begin{align}\label{dp-mainthb-conalcurve}
            l(0,\omega)=l(2\pi,\omega)~~\text{and}~~
            \frac{\mathrm{d}}{\mathrm{d}\tau}l(\tau,\omega)\in \mathcal{C}_{M}^{\omega}(l(\tau,\omega)),~~\text{for~any}~\tau\in[0,2\pi].
        \end{align}
    \end{theorem}
    
    An immediate corollary of Theorems \ref{mainth-a}-\ref{mainth-b} is the following
    
    \begin{corollary}\label{dp-coro-do} 
    Let all hypotheses hold in Theorem \ref{mainth-a}. Then the  unidirectional coupling stochastic system \eqref{intro-eqn-psi} admits a dynamical order, i.e., system \eqref{intro-eqn-psi} admits a simple asymptotic one-dimensional structure which is totally ordered with respect to the standard order ``$\le$" in $\mathbb{R}^{N}$.
    \end{corollary}
    \begin{proof}
        Let $\mathscr{A}$ be the unwrapping of $\mathcal{A}$ in Theorem \ref{mainth-a} from $M$ to $\mathbb{R}^{N}$, that is, 
	\begin{align}\label{dp-coro-eqnscrA}
		\mathscr{A}(\omega)=\left\{x\in\mathbb{R}^{N}\mid x\,\bmod\,2\pi\mathbf{e}\in\mathcal{A}(\omega)\right\},\quad\omega\in\Omega.
	\end{align}
        By virtue of Theorem \ref{mainth-a}, for each $\omega\in\Omega$, $\mathscr{A}(\omega)$ is a one-dimensional curve in $\mathbb{R}^{N}$. Due to the expression of $\mathcal{A}$ in Theorem \ref{mainth-b}, one can give the 
        parametrization $\mathscr{A}(\omega)=\{\ell(\tau,\omega)\mid \tau\in\mathbb{R}\}$ with
        \begin{align}\label{dp-coro-eqnperiod}
            \ell(\tau+2\pi,\omega)=\ell(\tau,\omega)+2\pi\mathbf{e},\quad \tau\in\mathbb{R},\,\omega\in\Omega.
        \end{align}
        We now define a one-dimensional $\mathbb{R}^{N}$-valued random set $\mathscr{B}$ as
        \begin{align*}
            \mathscr{B}\colon\omega\mapsto\mathscr{B}(\omega)=\left\{x+z(\omega)\mid x\in\mathscr{A}(\omega)\right\},~~\text{for}~~\omega\in\Omega,
        \end{align*}
        where the random variable $z$ is from Ornstein-Uhlenbeck process given in \eqref{prly-eqn-z}. 
        
        Firstly, we show that the random set $\mathscr{B}$ is globally attractive in the sense that, for any $\mathbb{R}^{N}$-valued  tempered random set $D$,
        \begin{align}\label{dp-coro-psiconverge}
            \lim_{t\to \infty}\operatorname{dist}\big(\psi\left(t,D(\theta_{-t}\omega),\theta_{-t}\omega\right),\mathscr{B}(\omega)\big)=0~~\text{for}~~ \omega\in\Omega.
        \end{align}
       For this purpose, we define $U(\omega)=\{x-z(\omega)\mid x\in D(\omega)\}$ for each $\omega\in\Omega$.  Due to \eqref{prly-eqn-psiphiOU} and the translation invariance of the Hausdorff semi-distance, one needs to show
        \begin{align}\label{dp-coro-eqn1}
            \lim_{t\to\infty}\operatorname{dist}\big(\phi\left(t,U(\theta_{-t}\omega),\theta_{-t}\omega\right),\mathscr{A}(\omega)\big)=0 ~~\text{for}~~ \omega\in\Omega.
        \end{align}
        For such an $\mathbb{R}^{N}$-valued tempered set $U$, there exist some tempered random variable $r\colon\Omega\to\mathbb{R}_{+}$ and $q\in\mathbb{R}^{N}$ such that $U(\omega)\subset\{x\in\mathbb{R}^{N}\mid \|x-q\|\leq r(\omega)\}$ for all $\omega\in\Omega$. Let $L_q=\{q+\lambda\cdot 2\pi\mathbf{e}\mid\lambda\in \mathbb{R}\}\subset\mathbb{R}^N$ be the straight-line through $q$ along the direction $2\pi\mathbf{e}$, and define
        \begin{align*}
            V(\omega)=\left\{x\in\mathbb{R}^{N}\;\middle|\; d_{\mathbb{R}^{N}}(x,L_{q})\leq r(\omega)\right\},\quad\omega\in\Omega.
        \end{align*}
        Clearly, $U(\omega)\subset V(\omega)$ for $\omega\in\Omega$, 
        and $\hat{V}\triangleq V\,\bmod\,2\pi\mathbf{e}$ is an $M$-valued tempered set. Set $\mathscr{V}^{t}(\omega)\triangleq\phi\left(t,V(\theta_{-t}\omega),\theta_{-t}\omega\right)\subset\mathbb{R}^{N}$ and $\mathcal{V}^{t}(\omega)\triangleq\Phi\left(t,\hat{V}(\theta_{-t}\omega),\theta_{-t}\omega\right)\allowbreak\subset M$. Then for any $\omega\in\Omega$, $y\in\mathscr{V}^{t}(\omega)$ and $\hat{y}\triangleq y\mod2\pi\mathbf{e}\in \mathcal{V}^{t}(\omega)$, 
        \begin{align}\label{dp-coro-distance}
            d_{\mathbb{R}^{N}}(y,\mathscr{A}(\omega))=\inf_{x\in\mathscr{A}(\omega),n\in\mathbb{Z}}\|y-x-n\cdot2\pi\mathbf{e}\|=d_{M}\big(\hat{y},\mathcal{A}(\omega)\big),
        \end{align}
        where the first equality follows from $\mathscr{A}(\omega)=\{x+n\cdot2\pi\mathbf{e}\mid x\in\mathscr{A}(\omega),n\in\mathbb{Z}\}$ (see \eqref{dp-coro-eqnscrA}) and the second equality follows from the definition of $d_{M}$ (see \eqref{prly-eqn-dM}). Let $\mathscr{U}^{t}(\omega)\triangleq\phi\left(t,U(\theta_{-t}\omega),\theta_{-t}\omega\right)\subset\mathbb{R}^{N}$. Due to \eqref{dp-coro-distance}, one has
        \begin{align*}
            \operatorname{dist}\big(\mathscr{U}^{t}(\omega),\mathscr{A}(\omega)\big)\leq\operatorname{dist}\big(\mathscr{V}^{t}(\omega),\mathscr{A}(\omega)\big)=\operatorname{dist}\big(\mathcal{V}^{t}(\omega),\mathcal{A}(\omega)\big),\quad \omega\in\Omega.
        \end{align*}
        Therefore, it follows from \eqref{dp-maintha-attract} in Theorem \ref{mainth-a} that   $\operatorname{dist}\big(\mathscr{U}^{t}(\omega),\mathscr{A}(\omega)\big)\to0$ as $t\to \infty$ for $\omega\in\Omega$, that is, \eqref{dp-coro-eqn1} holds; and hence, we have proved \eqref{dp-coro-psiconverge}.

        Next, we show that $\mathscr{B}$ is totally ordered with respect to the standard order ``$\le$" in $\mathbb{R}^{N}$, that is, for each $\omega\in\Omega$ and any two distinct points $b_{1},b_{2}\in\mathscr{B}(\omega)$, $b_{1}-b_{2}\in\mathbb{R}_{+}^{N}\cup (-\mathbb{R}_{+}^{N})$. In fact, by combining with \eqref{dp-mainthb-conalcurve} in Theorem \ref{mainth-b}, we deduce from \eqref{dp-coro-eqnperiod} that, for each $\omega\in \Omega$, $\mathscr{A}(\omega)$ is a smooth curve and satisfies
        \begin{align}\label{dp-coro-inC}
            \frac{\mathrm{d}}{\mathrm{d}\tau}\ell(\tau,\omega)\in\mathscr{C},~~\text{for}~~\tau\in\mathbb{R},
        \end{align} 
        where  $\mathscr{C}\subset\mathbb{R}_{+}^{N}$ is the closed convex cone as in \eqref{dp-propdp-eqnconeC}. So, for any two points $b_{1},b_{2}\in\mathscr{B}(\omega)$, one has  $b_{i}=\ell(u_{i},\omega)+z(\omega)$, for some $u_{i}\in\mathbb{R}$ and $i=1,2$. Without loss of generality, we assume that $u_{1}<u_{2}$. It then follows from \eqref{dp-coro-inC} that
        \begin{align*}
            b_{2}-b_{1}&=\ell(u_{2},\omega)-\ell(u_{1},\omega)\\
            &\left.=\int_{0}^{1}\frac{\mathrm{d}}{\mathrm{d}\tau}\ell(\tau,\omega)\right|_{\tau=\lambda u_{2}+(1-\lambda)u_{1}}(u_{2}-u_{1})\,\mathrm{d}\lambda\in\mathscr{C}\subset\mathbb{R}^{N}_{+}.
        \end{align*}
        Hence, we have proved that $\mathscr{B}$ is totally ordered with respect to the standard order ``$\le$" in $\mathbb{R}^{N}$, which completes the proof of Corollary \ref{dp-coro-do}.
    \end{proof}

    \section{Proof of Theorem \ref{mainth-b}}\label{section-pfmainth-b}
        We will prove Theorems \ref{mainth-a}-\ref{mainth-b} in the following two sections. In this section, we will prove Theorem \ref{mainth-b} under the assumption that Theorem \ref{mainth-a} holds. The proof of Theorem \ref{mainth-a} will be postponed in Section \ref{section-pfmainth-a}.
        \begin{proof}[Proof of \rm Theorem \ref{mainth-b}]
            Choose a closed and tempered
           random conal curve $l^{0}$ with respect to $\mathcal{C}_{M}$,
            given by $l^{0}\colon[s_{0},s_{1}]\times\Omega\to M$; $(s,\omega)\mapsto l^{0}(s,\omega)$ with $l^{0}(s_{0},\omega)=l^{0}(s_{1},\omega),$ satisfying 
            \begin{align*}
                0 \neq
                \frac{\mathrm{d}}{\mathrm{d}s}l^{0}(s,\omega)\in \mathcal{C}_{M}^{\omega}(l^{0}(s,\omega))
            \end{align*}for each $s\in[s_{0},s_{1}]$ and $\omega\in\Omega$. Let $\ell^{0}\colon \mathbb{R}\times\Omega\to\mathbb{R}^{N};\,(s,\omega)\mapsto \ell^{0}(s,\omega)$ be the unwrapping of $l^{0}$ from $M$ to $\mathbb{R}^{N}$. Then, for each $\omega\in\Omega$ and $s\in\mathbb{R}$, $\ell^{0}$ satisfies $0\neq\frac{\mathrm{d}}{\mathrm{d}s}\ell^{0}(s,\omega)\in\mathscr{C}$ by the definition of $\mathscr{C}$ (see \eqref{dp-propdp-eqnconeC}),  and
            \begin{align*}
			\ell^{0}\Big(s+(s_{1}-s_{0}),\omega\Big)=\ell^{0}(s,\omega)+2\pi\mathbf{e},
		\end{align*}
		since $l^{0}(s_{0},\omega)=l^{0}(s_{1},\omega)$. Noticing that $\mathscr{C}\setminus\{0\}\subset \operatorname{Int}\mathbb{R}^{N}_{+}$, one has 
        \begin{align}\label{eq-l' inintRN+}
           \frac{\mathrm{d}}{\mathrm{d}s}\ell^{0}(s,\theta_{-t}\omega)\in\operatorname{Int}\mathbb{R}^{N}_{+},~~\text{for~any}~  s\in\mathbb{R},\,t\geq0,\,\omega\in\Omega.
        \end{align}
    Given each $t\geq0$ and $\omega\in\Omega$, we define  
    \begin{align}\label{do-pfthb-Et}
        \mathcal{E}^{t}(\omega)=\left\{l^{t}(s,\omega)\triangleq\Phi\Big(t,l^{0}(s,\theta_{-t}\omega),\theta_{-t}\omega\Big)\;\middle| \;s\in[s_{0},s_{1}]\right\}\subset M
    \end{align}
    and let $\mathscr{E}^{t}(\omega)$ be the unwrapping of $\mathcal{E}^{t}(\omega)$ from $M$ to $\mathbb{R}^{N}$ as 
    \begin{align*}
      \mathscr{E}^{t}(\omega)=\left\{\ell^{t}(s,\omega)\triangleq\phi\Big(t,\ell^{0}(s,\theta_{-t}\omega),\theta_{-t}\omega\Big)\;\middle|\; s\in\mathbb{R}\right\}\subset\mathbb{R}^{N}.
    \end{align*}
    Clearly, one has $l^{t}(s_{0},\omega)=l^{t}(s_{1},\omega)$ and
    \begin{align}\label{do-pfthb-prdells}
           \ell^{t}\Big(s+(s_{1}-s_{0}),\omega\Big)=\ell^{t}(s,\omega)+2\pi\mathbf{e},\quad s\in\mathbb{R}.
    \end{align}

    Recall that $\Phi$-forward time evolution of random conal curve $l^{0}$ is also a random conal curve (see Proposition \ref{dp-prop-conalcurve}). Then, one has 
    \begin{align*}
        \frac{\mathrm{d}}{\mathrm{d}s}l^{t}(s,\omega)\in \mathcal{C}_{M}^{\omega}(l^{t}(s,\omega)),~~\text{for}~~ s\in[s_{0},s_{1}],
    \end{align*} 
    which entails that
    \begin{align}\label{do-pfthb-inC}
        \frac{\mathrm{d}}{\mathrm{d}s}\ell^{t}(s,\omega)\in\mathscr{C},~~\text{for}~~  s\in\mathbb{R}.
    \end{align} 
     Hence,  $\frac{\mathrm{d}}{\mathrm{d}s}\ell^{t}(s,\omega)\in\operatorname{Int}\mathbb{R}^{N}_{+}$, for $s\in\mathbb{R}$, a fact that follows from \eqref{eq-l' inintRN+} and the positive invariance of $\operatorname{Int}\mathbb{R}^{N}_{+}$ with respect to the linearized system $\delta\phi$ (see \eqref{hc-prophc-eqnzeta}) as shown in Lemma \ref{prly-lem-monotone}. Notably, $\frac{\mathrm{d}}{\mathrm{d}s}\ell_{1}^{t}(s,\omega)>0$ for $s\in\mathbb{R}$, where $\ell_{1}^{t}(s,\omega)$ is the first component of $\ell^{t}(s,\omega)$. Moreover, \eqref{do-pfthb-prdells} implies that the range of $\ell_{1}^{t}(\cdot,\omega)$ is $\mathbb{R}$. For simplicity, we write the map $\tau\colon\mathbb{R}\to\mathbb{R};s\mapsto\tau(s)\triangleq\ell_{1}^{t}(s,\omega)$. Then $\tau=\tau(s)$ is a bijective map. We further write $s=s(\tau)$ as its inverse. So, $\ell^{t}(s,\omega)$ can be re-parameterized as 
     \begin{equation}\label{do-pfthb-tildeell}
         {\fontsize{8.2}{9.6}\selectfont\begin{aligned}
            \tilde{\ell}^{t}(\tau,\omega)\triangleq\ell^{t}\Big(s(\tau),\omega\Big)=\Big(\ell_{1}^{t}(s,\omega),\ell_{2}^{t}(s,\omega),\ldots,\ell_{N}^{t}(s,\omega)\Big)^{\top}=\Big(\tau,\tilde{\ell}^{t}_{2}(\tau,\omega),\ldots,\tilde{\ell}^{t}_{N}(\tau,\omega)\Big)^{\top},
        \end{aligned}}
     \end{equation}
        for $\tau\in\mathbb{R}$. Furthermore, together with \eqref{do-pfthb-prdells}, we have
        \begin{align*}
            \tau+2\pi=\ell_{1}^{t}(s,\omega)+2\pi=\ell_{1}^{t}\Big(s+(s_{1}-s_{0}),\omega\Big)=\tau\Big(s+(s_{1}-s_{0})\Big),
        \end{align*}
        which entails that
        \begin{align}\label{do-pfthb-iota}
             s(\tau+2\pi)=s(\tau)+(s_{1}-s_{0}).
        \end{align}
    Consequently,         {\small\begin{align*}
            \tilde{\ell}^{t}(\tau+2\pi,\omega)\stackrel{\eqref{do-pfthb-tildeell}}{=}\ell^{t}\Big(s(\tau+2\pi),\omega\Big)&\stackrel{\eqref{do-pfthb-iota}}{=}\ell^{t}\Big(s(\tau)+(s_{1}-s_{0}),\omega\Big)\\
         &\stackrel{\eqref{do-pfthb-prdells}}{=}\ell^{t}\Big(s(\tau),\omega\Big)+2\pi\mathbf{e}\stackrel{\eqref{do-pfthb-tildeell}}{=}\tilde{\ell}^{t}(\tau,\omega)+2\pi\mathbf{e}
        \end{align*}}for any $\tau\in\mathbb{R}$, and one can define a map on $M$ as $\tilde{l}^{t}(\cdot,\omega)\colon[0,2\pi]\to M;\,\tau\mapsto\tilde{l}^{t}(\tau,\omega)$:
        \begin{align*}
            \tilde{l}^{t}(\tau,\omega)\,\triangleq\,\tilde{\ell}^{t}(\tau,\omega)\,\bmod\,2\pi\mathbf{e},\quad\tau\in[0,2\pi].
        \end{align*}
        Clearly, $\tilde{l}^{t}(0,\omega)=\tilde{l}^{t}(2\pi,\omega)$; and moreover, by virtue of the expression for $\mathcal{E}^{t}(\omega)$ in \eqref{do-pfthb-Et}, we obtain
        \begin{align}\label{do-pfthb-eqn1}
            \mathcal{E}^{t}(\omega)=\left\{\tilde{l}^{t}(\tau,\omega)\;\middle|\; \tau\in[0,2\pi]\right\}\subset M.
        \end{align}
        
        We now prove that for each $\omega\in\Omega$, the family of $M$-valued maps $\{\tilde{l}^{t}(\tau,\omega)\}_{t\geq0}$ 
        is equicontinuous in $\tau\in[0,2\pi]$. As a matter of fact, for $t\geq0$ and $\tau\in[0,2\pi]$, it follows from \eqref{do-pfthb-inC} that 
        \begin{align}\label{do-pfthb-tildeinC}
            \frac{\mathrm{d}}{\mathrm{d}\tau}\tilde{\ell}^{t}(\tau,\omega)\in\mathscr{C};
        \end{align}
        and hence, one has
        \begin{align}\label{do-pfthb-equi}
            \gamma_{1}\cdots\gamma_{i-1}\leq\frac{\mathrm{d}}{\mathrm{d}\tau}\tilde{\ell}_{i}^{t}(\tau,\omega)\leq\beta_{1}\cdots\beta_{i-1},~~\text{for}~~ i=2,\ldots,N.
        \end{align}
        Together with the fact that $d_{M}\Big(\tilde{l}^{t}(\tau_{2},\omega),\tilde{l}^{t}(\tau_{1},\omega)\Big)\leq\left\|\tilde{\ell}^{t}(\tau_{2},\omega)-\tilde{\ell}^{t}(\tau_{1},\omega)\right\|$ for $\tau_{1},\tau_{2}\allowbreak\in[0,2\pi]$, \eqref{do-pfthb-equi} entails that the family of $M$-valued maps $\{\tilde{l}^{t}(\tau,\omega)\}_{t\geq0}$ is equicontinuous in $\tau\in[0,2\pi]$. 
        
        We further assert that the family of $M$-valued maps $\{\tilde{l}^{t}(\tau,\omega)\}_{t\geq 0}$ is uniformly bounded on $\tau\in[0,2\pi]$. In fact, by virtue of Theorem \ref{mainth-a}, we have
		\begin{align}\label{ator-thm-eqnator}
			\operatorname{dist}\big(\mathcal{E}^{t}(\omega),\mathcal{A}(\omega)\big)\to0,~~\text{as}~~ t\to\infty.
	\end{align} 
        So, by \eqref{do-pfthb-eqn1}, there exists $T(\omega)>0$ such that for $t\geq T(\omega)$,
        \begin{align*}
            \tilde{l}^{t}(\tau,\omega)\in 
            \{x\in M\mid d_{M}\big(x,\mathcal{A}(\omega)\big)\leq1\},~~\text{for~any}~~\tau\in[0,2\pi].
        \end{align*} 
        Hence, we proved the assertion.

        Now, due to the equicontinuity and uniform boundedness of  $\{\tilde{l}^{t}(\tau,\omega)\}_{t\geq 0}$, by the Arzel\'{a}-Ascoli theorem, one can find a subsequence $\{t_{m}(\omega)\}_{m\in\mathbb{N}}$ with $t_{m}(\omega)\to\infty$ ($m\to\infty$), and a map $l\colon[0,2\pi]\times\Omega\to M$  satisfying $l(0,\omega)=l(2\pi,\omega)$, such that
	\begin{align}\label{do-pfthb-convergence}
          \sup_{\tau\in[0,2\pi]}d_M\Big(\tilde{l}^{t_{m}(\omega)}(\tau,\omega),l(\tau,\omega)\Big)\to0,~~\text{as}~~ m\to\infty.	
	\end{align}
    Define $\mathcal{G}(\omega)\triangleq\left\{l(\tau,\omega)\;\middle|\;\tau\in[0,2\pi]\right\}$. Clearly, for each $\omega\in\Omega$, $\mathcal{G}(\omega)$ is a closed curve which is homeomorphic to $\mathbb{S}^{1}$.

    Next, we will show that, for each $\omega\in\Omega$,  $\mathcal{A}(\omega)=\mathcal{G}(\omega)$ whenever $k$ is sufficiently large. In fact, by \eqref{do-pfthb-eqn1} and \eqref{do-pfthb-convergence}, we have
    \begin{align*}
		\operatorname{dist}\big(\mathcal{G}(\omega),\mathcal{E}^{t_{m}(\omega)}(\omega)\big)&\;\,=\sup_{\tau\in[0,2\pi]}d_{M}\big(l(\tau,\omega),\mathcal{E}^{t_{m}(\omega)}(\omega)\big)\\
        &\stackrel{\eqref{do-pfthb-eqn1}}{=}\sup_{\tau\in[0,2\pi]}\inf_{u\in[0,2\pi]}d_{M}\Big(l(\tau,\omega),\tilde{l}^{t_{m}(\omega)}(u,\omega)\Big)\\
        &\!\stackrel{\eqref{do-pfthb-convergence}}{\leq}\sup_{\tau\in[0,2\pi]}d_{M}\Big(l(\tau,\omega),\tilde{l}^{t_{m}(\omega)}(\tau,\omega)\Big)\to0,~~\text{as}~~m\to\infty.
    \end{align*}
    Then, together with \eqref{ator-thm-eqnator}, one has $\operatorname{dist}\big(\mathcal{G}(\omega),\mathcal{A}(\omega)\big)=0$; and hence, $\mathcal{G}(\omega)\subset\mathcal{A}(\omega)$ for any $\omega\in\Omega$. Suppose that $\mathcal{G}(\omega)\subsetneq\mathcal{A}(\omega)$ for some $\omega\in\Omega$. Then there exists $x(\omega)\in\mathcal{A}(\omega)$ such that $\mathcal{G}(\omega)\subset\mathcal{A}(\omega)\setminus\{x(\omega)\}$. Since $\mathcal{A}(\omega)$ is homeomorphic to $\mathbb{S}^{1}$ (see Theorem \ref{mainth-a}), it follows that $\mathcal{A}(\omega)\setminus\{x(\omega)\}$ is homeomorphic to $\mathbb{R}$. Recall that $\mathcal{G}(\omega)$ is homeomorphic to $\mathbb{S}^{1}$. This contradicts that $\mathcal{G}(\omega)\subset\mathcal{A}(\omega)\setminus\{x(\omega)\}$. Thus, we have proved $\mathcal{A}(\omega)=\mathcal{G}(\omega)$ for any $\omega\in\Omega$. 
    
    Finally, we show that $\mathcal{A}$ is a random conal curve. This can be done by \eqref{do-pfthb-convergence} and the fact that $\frac{\mathrm{d}}{\mathrm{d}\tau}\tilde{l}^{t_{m}(\omega)}(\tau,\omega)\in\mathcal{C}_{M}^{\omega}(\tilde{l}^{t_{m}(\omega)}(\tau,\omega))$, which follows from \eqref{do-pfthb-tildeinC} and the definition of $\mathcal{C}_{M}$. Hence, $\frac{\mathrm{d}}{\mathrm{d}\tau}l(\tau,\omega)\in\mathcal{C}_{M}^{\omega}(l(\tau,\omega))$, for any $\tau\in[0,2\pi]$. So, we have obtained that $\mathcal{A}(\omega)=\{l(\tau,\omega)\mid\tau\in[0,2\pi]\}$ is a closed random curve on $M$. Thus, we have completed the proof. 
    \end{proof}

	\begin{remark}\label{do-rmk-difference}
            In order to obtain the dynamical order mentioned in Corollary \ref{dp-coro-do}, Chow et al. (see \cite[p.1017]{ChowShenZhou07}) constructed a random family of so-called horizontal curves by solving the following inequalities:
            \begin{align}\label{hc-rmkhc-Shenidea1a}
                \frac{\mathrm{d}(\delta \phi)_{i+1}/\mathrm{d}t}{\mathrm{d}(\delta \phi)_{i}/\mathrm{d}t}< \beta_{i}~~\text{and}~~\frac{\mathrm{d}(\delta \phi)_{i}}{\mathrm{d}t} > 0,~~\text{whenever}~~ (\delta \phi)_{i+1}=\beta_{i}(\delta \phi)_{i},
            \end{align}
            and 
            \begin{align}\label{hc-rmkhc-Shenidea1b}
                \frac{\mathrm{d}(\delta \phi)_{i}/\mathrm{d}t}{\mathrm{d}(\delta x)_{i+1}/\mathrm{d}t} < \frac{1}{\gamma_{i}}~~\text{and}~~\frac{\mathrm{d}(\delta \phi)_{i+1}}{\mathrm{d}t}>0,~~\text{whenever}~~(\delta \phi)_{i+1}=\gamma_{i}(\delta \phi)_{i},
            \end{align} 
            with pairs $(\beta_{i},\gamma_{i})$ satisfying $0<\gamma_{i}<1$ and $\beta_{i}>1$, for $i=1,\ldots,N-1$. Here $\delta \phi$ is the variation of $\phi$ (the differential equation for $\phi$ is given in \eqref{prly-eqn-phi}), and $(\delta \phi)_{i}$ represents the $i$-th component of $\delta\phi$, for $i=1,\cdots,N-1$. Due to the interaction of the $i$-th oscillator with both its successor and predecessor, the second inequalities in both \eqref{hc-rmkhc-Shenidea1a} and  \eqref{hc-rmkhc-Shenidea1b} naturally hold for every pair $(\beta_{i},\gamma_{i})$.  Consequently, the first inequalities in both  \eqref{hc-rmkhc-Shenidea1a} and  \eqref{hc-rmkhc-Shenidea1b} 
           can be solved for certain appropriate  $(\beta_{i},\gamma_{i})$ with $0<\gamma_{i}<1$ and $\beta_{i}>1$, $i=1,\ldots,N-1$. Based on this, the construction of horizontal curves was accomplished  in their works.

            Unfortunately, the positivity of the second inequality in \eqref{hc-rmkhc-Shenidea1a} is \textit{undermined} when one encounters the unidirectional coupling depicted in Figure \ref{intro-figline-uni}, since there is no control from successors at all. To overcome such difficulties, we revisited \eqref{hc-rmkhc-Shenidea1a} and \eqref{hc-rmkhc-Shenidea1b} from the perspective of differential positivity, and \textit{found that horizontal curves can be identified as the conal curves on $M$}; and moreover, the construction of the conal curves amounts to finding some appropriate $(\beta_{i},\gamma_{i})$ such that the variational system of $\phi$ leaves the cone $\mathscr{C}$ invariant, where $\mathscr{C}$ is given in \eqref{dp-propdp-eqnconeC} as
            \begin{align*}
			\mathscr{C}=\{\delta x\in\mathbb{R}^{N}_{+}\mid\gamma_{i}{\delta x}_{i}\leq {\delta x}_{i+1}\leq\beta_{i}{\delta x}_{i},i=1,\ldots,N-1\}.
		\end{align*}
            This is achieved by the geometric cone criteria \eqref{dp-propdp-eqnbeta} and \eqref{dp-propdp-eqngamma}, which notably do not involve the estimates for the sign of variational system of $\phi$.  
        \end{remark}

    \section{Proof of Theorem \ref{mainth-a}}\label{section-pfmainth-a} 
    In this section, we focus on the proof of Theorem \ref{mainth-a}. Our approach is motivated by \cite{ChowShenZhou07}, however, the presence of unidirectional coupling makes the coupling matrix $L$ non-diagonalizable, thereby making the realization of the computational procedure described in \cite{ChowShenZhou07} more compliated, particularly for the constructions of global random attractor $\mathcal{A}$ of random system $\Phi$ (see Subsection \ref{subsection-pftha-1}), as well as those of the center-unstable manifold and its foliation for random system $\phi$ more involved (see Subsection \ref{subsection-pftha-2}).

    Moreover, in order to show that $\mathcal{A}(\omega)$ is a closed smooth curve for each $\omega\in \Omega$, we first establish that for coupling coefficient $k$ sufficiently large, the random system $\phi$ on $\mathbb{R}^{N}$ possesses a one-dimensional and $C^{1}$-smooth center-unstable manifold with a stable foliation (see Proposition \ref{mfld-prop-mfld} and Lemma \ref{do-lem-foliation}). Based on such a foliation structure, we prove that the unwrapping $\mathscr{A}$ of $\mathcal{A}$ from $M$ to $\mathbb{R}^{N}$ (see \eqref{dp-coro-eqnscrA}) actually coincides with the obtained center-unstable manifold in Proposition \ref{mfld-prop-mfld}, by which we accomplish obtaining that  $\mathcal{A}(\omega),\omega\in\Omega$, is a closed smooth curve.

	\subsection{Global Random Attractor \texorpdfstring{$\mathcal{A}$}{mathcal{A}} on \texorpdfstring{$M$}{M}}\label{subsection-pftha-1}
    We will prove in this subsection the existence of global random attractor $\mathcal{A}$ of random system $\Phi$ on $M$. Before proceeding this, we present some estimates for later use. Based on \eqref{prly-eqn-H} and Lemma \ref{pre-lem-hi}(ii), denote  
	\begin{equation*}
		C_{1}(\omega)=\underset{t\in\mathbb{R},1\leq i\leq N}{\sup}\frac{\vert H_{i}(\theta_{t}\omega)\vert}{1+\vert t\vert}<\infty,~~\text{for}~~\omega\in\Omega.
	\end{equation*}
	Then, for any $t\in\mathbb{R}$ and $\omega\in\Omega$, one has
	\begin{equation}\label{ator-pfprop1-eqnC1}
		C_{1}(\theta_{t}\omega)\leq C_{1}(\omega)(1+\vert t\vert)
	\end{equation}
and
\begin{equation}\label{ator-pfprop1-eqnH}
		\vert H_{i}(\theta_{t}\omega)\vert\leq C_{1}(\omega)(1+\vert t\vert),~~\text{for}~~ i=1,\ldots,N.
    \end{equation}
    Clearly, $C_{1}$ is $(\mathcal{F},\mathcal{B}(\mathbb{R}_{+}))$-measurable due to the $(\mathcal{F},\mathcal{B}(\mathbb{R}))$-measurability of $\omega\mapsto H_{i}(\theta_{t}\omega)$ for each $t\in\mathbb{R}$ and the continuity of  $t\mapsto H_{i}(\theta_{t}\omega)$ for each $\omega\in\Omega$ and $i=1,\ldots,N$.

    \begin{lemma}\label{ator-lem-exist}
		The random system $\Phi$ on $M$ possesses a global random attractor $\mathcal{A}\colon\Omega\to2^{M}$. 
    \end{lemma}
	
    \begin{proof}
        Based on estimates of each part of $\phi$ defined in system \eqref{prly-eqn-phi}, we first show that random system $\Phi$ possesses a $\mathcal{D}$-absorbing random compact set, where $\mathcal{D}$ is a family of $M$-valued tempered random sets. For this purpose, with any arbitrary $\omega\in\Omega$ fixed, let $D$ be an $\mathbb{R}^{N}$-valued random set such that $\hat{D}\triangleq D\,\bmod\,2\pi\mathbf{e}\in\mathcal{D}$ and $\phi^{0}=(\phi^{0}_{1},\ldots,\phi^{0}_{N})^{\top}\in D(\theta_{-t}\omega)$ for some $t\geq0$. The solution $\phi(t,\phi^{0},\theta_{-t}\omega)$ to random system \eqref{prly-eqn-phi} with initial data $\phi^{0}$ is denoted by $(\phi_{1},\ldots,\phi_{N})^{\top}$, where $\phi_{i}$ is the $i$-th component of $\phi(t,\phi^{0},\theta_{-t}\omega)$ for $i=1,\ldots,N$. By the variation of constants formula, one has
		\begin{equation}\label{ator-eqn-varconphi}
			\phi(t,\phi^{0},\theta_{-t}\omega)=e^{k Lt}\phi^{0}+\int_{0}^{t}e^{kL(t-s)}F(\phi(s,\phi^{0},\theta_{-t}\omega),\theta_{s-t}\omega)\,\mathrm{d}s,
		\end{equation}
    		where $e^{kLt}$ is a lower triangular matrix with its elements satisfying $(e^{kLt})_{1,1}=1$ and 
        \begin{align}\label{ator-eqn-ekLt}
            (e^{kLt})_{i,j}=\begin{cases}
                       1-e^{-kt}\sum_{m=0}^{i-2}\frac{(kt)^m}{m!},&2\leq i\leq N,j=1,\\[8pt]
                       e^{-kt}\frac{(kt)^{i-j}}{(i-j)!},&2\leq j\leq i\leq N,
                       \end{cases}
        \end{align}  
    and
         \begin{align}\label{ator-eqn-F}
		F(\phi(s,\phi^{0},\theta_{-t}\omega),\theta_{s-t}\omega)&=\left( \begin{array}{c}
			F_{1}(\phi(s,\phi^{0},\theta_{-t}\omega),\theta_{s-t}\omega)  \\
			\cdots \\
			 F_{N}(\phi(s,\phi^{0},\theta_{-t}\omega),\theta_{s-t}\omega) 
		\end{array} \right)
	\end{align}
        satisfying $F_{i}(\phi(s,\phi^{0},\theta_{-t}\omega),\theta_{s-t}\omega)\triangleq\alpha_{i} - \sin\left(\phi_{i} + z_{i}(\theta_{s-t}\omega)\right)+H_{i}(\theta_{s-t}\omega)$. Then it follows from \eqref{ator-eqn-varconphi}-\eqref{ator-eqn-F} that 
        \begin{equation}\label{ator-pflem1-minphi21}
        {\fontsize{9}{10.8}\selectfont\begin{aligned}
           |\phi_{2}-\phi_{1}|&\leq e^{-kt}|\phi_{1}^{0}|+e^{-kt}|\phi_{2}^{0}|+\int_{0}^{t}e^{-k(t-s)}\left|F_{1}(\phi(s,\phi^{0},\theta_{-t}\omega),\theta_{s-t}\omega)\right|\,\mathrm{d}s\\
            &\quad+\int_{0}^{t}e^{-k(t-s)}\left|F_{2}(\phi(s,\phi^{0},\theta_{-t}\omega),\theta_{s-t}\omega)\right|\,\mathrm{d}s
        \end{aligned}}
        \end{equation}and
        \begin{equation}\label{ator-pflem1-minphi}
            {\fontsize{8}{9.6}\selectfont\begin{aligned}
                |\phi_{i+1}-\phi_{i}|&\leq \frac{(kt)^{i-1}}{(i-1)!}e^{-kt}|\phi^{0}_{1}|+\sum_{j=1}^{i-1}\left[\frac{(kt)^{i-j}}{(i-j)!}+\frac{(kt)^{i-j-1}}{(i-j-1)!}\right]e^{-kt}|\phi^{0}_{j+1}|+e^{-kt}|\phi^{0}_{i+1}| \\
                &\quad+\int_{0}^{t}\frac{(k(t-s))^{i-1}}{(i-1)!}e^{-k(t-s)}\left|F_{1}(\phi(s,\phi^{0},\theta_{-t}\omega),\theta_{s-t}\omega)\right|\,\mathrm{d}s\\
                &\quad+\sum_{j=1}^{i-1}\int_{0}^{t}\frac{(k(t-s))^{i-j}}{(i-j)!}e^{-k(t-s)}\left|F_{j+1}(\phi(s,\phi^{0},\theta_{-t}\omega),\theta_{s-t}\omega)\right|\,\mathrm{d}s\\
                &\quad+\sum_{j=1}^{i-1}\int_{0}^{t}\frac{(k(t-s))^{i-j-1}}{(i-j-1)!}e^{-k(t-s)}\left|F_{j+1}(\phi(s,\phi^{0},\theta_{-t}\omega),\theta_{s-t}\omega)\right|\,\mathrm{d}s\\
                &\quad+\int_{0}^{t}e^{-k(t-s)}\left|F_{i+1}(\phi(s,\phi^{0},\theta_{-t}\omega),\theta_{s-t}\omega)\right|\,\mathrm{d}s
                \end{aligned}}
                \end{equation}for $i=2,\ldots,N-1$. Let $\alpha_{*}=1+\max_{1\leq i\leq N}|\alpha_{i}|$. Then, \eqref{ator-pfprop1-eqnH} implies 
        \begin{align}\label{ator-pflem1-Fj}
            \left|F_{j}(\phi(s,\phi^{0},\theta_{-t}\omega),\theta_{s-t}\omega)\right|\leq\alpha_{*}+C_{1}(\omega)(1+t-s)
        \end{align}
        for $j=1,\ldots,N,\,0\leq s\leq t.$ By combining with \eqref{ator-pflem1-minphi21}-\eqref{ator-pflem1-Fj}, we have
        \begin{align*}
            |\phi_{2}-\phi_{1}|\leq e^{-kt}|\phi_{1}^{0}|+e^{-kt}|\phi_{2}^{0}|+2\int_{0}^{t}e^{-k(t-s)}\big[\alpha_{*}+C_{1}(\omega)(1+t-s)\big]\,\mathrm{d}s
        \end{align*}
        and
        {\fontsize{8.2}{9.6}\selectfont\begin{align*}
            |\phi_{i+1}-\phi_{i}|&\leq\frac{(kt)^{i-1}}{(i-1)!}e^{-kt}|\phi^{0}_{1}|+\sum_{j=1}^{i-1}\left[\frac{(kt)^{i-j}}{(i-j)!}+\frac{(kt)^{i-j-1}}{(i-j-1)!}\right]e^{-kt}|\phi^{0}_{j+1}|+e^{-kt}|\phi^{0}_{i+1}|\\
            &\quad+2\sum_{m=0}^{i-1}\int_{0}^{t}\frac{(k(t-s))^{m}}{m!}e^{-k(t-s)}\big[\alpha_{*}+C_{1}(\omega)(1+t-s)\big]\,\mathrm{d}s\\
            &=\frac{(kt)^{i-1}}{(i-1)!}e^{-kt}|\phi^{0}_{1}|+\sum_{j=1}^{i-1}\left[\frac{(kt)^{i-j}}{(i-j)!}+\frac{(kt)^{i-j-1}}{(i-j-1)!}\right]e^{-kt}|\phi^{0}_{j+1}|+e^{-kt}|\phi^{0}_{i+1}|\\
            &\quad+2\sum_{m=0}^{i-1}\Bigg[\frac{\alpha_{*}+C_{1}(\omega)}{k}\bigg(1-e^{-kt} \sum_{u=0}^{m} \frac{(kt)^u}{u!} \bigg)+\frac{C_{1}(\omega)(m+1)}{k^2} \bigg(1-e^{-kt} \sum_{u=0}^{m+1} \frac{(kt)^u}{u!} \bigg)\Bigg]
        \end{align*}}for $i=2,\ldots,N-1$, where 
        {\fontsize{8}{9.6}\selectfont\begin{align*}
            &\int_{0}^{t}\frac{(k(t-s))^{m}}{m!}e^{-k(t-s)}[\alpha_{*}+C_{1}(\omega)(1+t-s)]\,\mathrm{d}s\\
            =&\frac{\alpha_{*}+C_{1}(\omega)}{k}\bigg(1-e^{-kt} \sum_{u=0}^{m} \frac{(kt)^u}{u!}\bigg)+\frac{C_{1}(\omega)(m+1)}{k^2} \bigg(1-e^{-kt} \sum_{u=0}^{m+1} \frac{(kt)^u}{u!} \bigg)
        \end{align*}}is used. Hence, for $i,j=1,\ldots,N$, one has
        \begin{align}\label{atorlem-eqn-polyest}
            |\phi_{j}-\phi_{i}|\leq p(t)e^{-kt}\left\|\phi^{0}\right\|+r(\omega),
        \end{align}
        where 
        \begin{align}\label{ator-lem-pt}
            p(t)=\sum_{m=0}^{N-2}2(N-m-1)\frac{(kt)^m}{m!},~~\text{for}~~t\geq0
        \end{align}
        and 
         \begin{align}\label{ator-lem-r}
            r(\omega)=2(N-1)\sum_{m=0}^{N-2}\left(\frac{\alpha_{*}+C_{1}(\omega)}{k}+\frac{C_{1}(\omega)(m+1)}{k^{2}}\right).
         \end{align}
         Moreover, by \eqref{ator-pfprop1-eqnC1}, one has $r(\theta_{t}\omega)\leq r(\omega)(1+|t|)$ for $t\in\mathbb{R}$, which entails the tempered-ness of $r$ in \eqref{ator-lem-r}. 
         
         Observe that $\phi(t,\phi^{0},\theta_{-t}\omega)\triangleq(\phi_{1},\ldots,\phi_{N})^{\top}\in\mathbb{R}^{N}$ can be orthogonally decomposed as
        \begin{align}\label{atorlem-eqn-decom}
            \left(\phi_{\rm ave},\ldots,\phi_{\rm ave}\right)^{\top}+\left(\phi_{1}-\phi_{\rm ave},\ldots,\phi_{N}-\phi_{\rm ave}\right)^{\top},
        \end{align}
        where $\phi_{\rm ave}\triangleq\frac{1}{N}\sum_{i=1}^{N}\phi_{i}$. Let $L_{0}=\{\lambda\cdot2\pi\mathbf{e}\mid\lambda\in\mathbb{R}\}\subset\mathbb{R}^{N}$ be the straight-line through $0\in\mathbb{R}^{N}$ along the direction $2\pi\mathbf{e}$. Then
        \begin{equation}\label{ator-lem-disest}
            \begin{aligned}
            d_{\mathbb{R}^{N}}\big(\phi(t,\phi^{0},\theta_{-t}\omega),L_{0}\big)
            &=\left\|\left(\phi_{1}-\phi_{\rm ave},\ldots,\phi_{N}-\phi_{\rm ave}\right)^{\top}\right\|\\
            &\leq\sqrt{N}\Big[p(t)e^{-kt}\left\|\phi^{0}\right\|+r(\omega)\Big],
        \end{aligned}
        \end{equation}where the equality and inequality follow from \eqref{atorlem-eqn-decom} and \eqref{atorlem-eqn-polyest}, respectively.

        Define $\hat{L}_{0}\triangleq L_{0}\,\bmod\,2\pi\mathbf{e}\subset M$, and $\hat{0}\triangleq0\,\bmod\,2\pi\mathbf{e}\in M$. Then for  $\Phi(t,\Phi^{0},\theta_{-t}\allowbreak\omega)=\phi(t,\phi^{0},\theta_{-t}\omega)\,\bmod\,2\pi\mathbf{e}\in M$ such that $\Phi^{0}\triangleq\phi^{0}\,\bmod\,2\pi\mathbf{e}\in \hat{D}(\theta_{-t}\omega)$, one has{\allowdisplaybreaks\small\begin{align*}
            d_{M}\big(\Phi(t,\Phi^{0},\theta_{-t}\omega),\hat{L}_{0}\big)
            &\stackrel{\eqref{prly-eqn-dM}}{=}\inf_{n\in\mathbb{Z}}\inf_{y\in L_{0}}\|\phi(t,\phi^{0},\theta_{-t}\omega)-y-n\cdot2\pi\mathbf{e}\|\\
            &\stackrel{\eqref{prly-eqn-phiperiod}}{=}\inf_{n\in\mathbb{Z}}d_{\mathbb{R}^{N}}\Big(\phi\big(t,\phi^{0}-n\cdot2\pi\mathbf{e},\theta_{-t}\omega\big),L_{0}\Big)\\
            &\!\stackrel{\eqref{ator-lem-disest}}{\leq}\inf_{n\in\mathbb{Z}}\sqrt{N}\Big[p(t)e^{-kt}\left\|\phi^{0}-n\cdot2\pi\mathbf{e}\right\|+r(\omega)\Big]\\
            &\stackrel{\eqref{prly-eqn-dM}}{=}\sqrt{N}p(t)e^{-kt}d_{M}(\Phi^{0},\hat{0})+\sqrt{N}r(\omega).
        \end{align*}}Now, the tempered-ness of $\hat{D}$ implies that  there exists $t_{\hat{D}}(\omega)>0$ such that
        \begin{align*}
            d_{M}\big(\Phi(t,\Phi^{0},\theta_{-t}\omega),\hat{L}_{0}\big)\leq 1+\sqrt{N}r(\omega),~~\text{for}~~ t\geq t_{\hat{D}}(\omega).
        \end{align*}
        Together with the tempered-ness of $r$ in \eqref{ator-lem-r}, we obtain that there exists an $M$-valued tempered random compact set $K$ such that
        \begin{align*}
            \Phi(t,\hat{D}(\theta_{-t}\omega),\theta_{-t}\omega)\subset K(\omega),~~\text{for}~~t\geq t_{\hat{D}}(\omega).
        \end{align*}
        This implies that $K$ is a $\mathcal{D}$-absorbing random compact taking values in $M$. Hence, by virtue of Lemma \ref{prly-thm-ator}, we obtain that random system $\Phi$ on $M$ possesses a global random attractor $\mathcal{A}\colon\Omega\to2^{M}$ with $\mathcal{A}\in\mathcal{D}$, which completes the proof.
    \end{proof}

    \subsection{Topological Structure of \texorpdfstring{$\mathcal{A}$}{mathcal{A}} on \texorpdfstring{$M$}{M}}\label{subsection-pftha-2}
   
    In this subsection, we focus on the topological structure of the global random attractor $\mathcal{A}$ on $M$. More precisely, we will show that, for each $\omega\in\Omega$, $\mathcal{A}(\omega)$ is a one-dimensional smooth closed curve. Our proof is divided into the following two steps: 
    \begin{itemize}
        \item[{\rm(i)}]To prove the existence of a one-dimensional center-unstable manifold $W^{cu}$ with respect to random system $\phi$ on $\mathbb{R}^{N}$ (see Proposition \ref{mfld-prop-mfld});
        \item[{\rm(ii)}]To construct the stable foliation of $W^{cu}$ (see Lemma \ref{do-lem-foliation}) and show that $W^{cu}$ coincides with  the unwrapping $\mathscr{A}$ of $\mathcal{A}$ from $M$ to $\mathbb{R}^{N}$ (see Proposition \ref{do-prop-mfld=ator}). So, the closedness of $\mathcal{A}$ follows directly from the one-dimensionality of $\mathscr{A}$ and the periodicity of random system $\phi$ defined by \eqref{prly-eqn-phi}. As a by-product, $\mathcal{A}$ is also a $C^{1}$-smooth curve on each fiber.
    \end{itemize}

    \begin{proposition}\label{mfld-prop-mfld}
        For the random system $\phi$ on $\mathbb{R}^{N}$, if the coupling coefficient $k$ is sufficiently large, then  there exists a one-dimensional center-unstable $C^{1}$-smooth manifold $W^{cu}$.
    \end{proposition}

    In order to prove Proposition \ref{mfld-prop-mfld}, we first collect some necessary prerequisites. It is easy to check that $F(\phi,\omega)$ (see \eqref{ator-eqn-F}) is Lipschitz continuous with respect to $\phi\in\mathbb{R}^{N}$ for each $\omega\in\Omega$. More precisely, for $x,y\in\mathbb{R}^{N}$, one has 
    \begin{align*}
        \|F(x,\omega)-F(y,\omega)\|\leq \|x-y\|,\quad \forall~\omega\in\Omega.
    \end{align*}
    Hence, its Lipschitz constant, denoted by $\operatorname{Lip}_{\phi}F$, satisfies $\operatorname{Lip}_{\phi}F\leq 1$. Moreover, it follows from  \eqref{ator-eqn-F} and \eqref{ator-pflem1-Fj} that 
	\begin{align}\label{ator-eqn-estimateF}
		\| F(\phi,\theta_{t}\omega)\|\leq \sqrt{N}\big[\alpha_{*}+C_{1}(\omega)(1+|t|)\big],\quad \forall~\phi\in\mathbb{R}^{N},\,t\in\mathbb{R},\,\omega\in\Omega.
	\end{align}

    To construct the center-unstable manifold $W^{cu}$, let us define the projections $\mathcal{P}^{c}$ and $\mathcal{P}^{s}\colon\mathbb{R}^{N}\to\mathbb{R}^{N}$ by
    \begin{align}\label{ator-eqn-PcPs}
        \mathcal{P}^{c} = \begin{pmatrix}
            1 & 0 & \cdots & 0 \\
            1 & 0 & \cdots & 0 \\
            \vdots & \vdots & \ddots & \vdots \\
            1 & 0 & \cdots & 0
        \end{pmatrix}~~\text{and}~~
        \mathcal{P}^{s}=\begin{pmatrix}
		0 &  & &  \\
		-1 & 1 & &  \\
            \vdots& &\ddots\\
            -1&  & & 1
	\end{pmatrix},
    \end{align}
        respectively. Obviously, $\mathcal{P}^{c}e^{kLt}=e^{kLt}\mathcal{P}^{c}$ and $\mathcal{P}^{s}e^{kLt}=e^{kLt}\mathcal{P}^{s}$ for $t\in\mathbb{R}$. Denote by $E^{c}=\mathcal{P}^{c}\mathbb{R}^{N}$ and $E^{s}=\mathcal{P}^{s}\mathbb{R}^{N}$ the images of  $\mathcal{P}^{c}$ and $\mathcal{P}^{s}$, respectively. So, $\mathbb{R}^{N}=E^{c}\oplus E^{s}$ and $E^{c}$ is one-dimensional. 
    \begin{lemma}For the projections $\mathcal{P}^{c}$ and $\mathcal{P}^{s}$, we have the following estimates:
		\begin{align}
			&\|e^{kLt}\mathcal{P}^{c}\|_{\mathscr{L}(\mathbb{R}^{N},\mathbb{R}^{N})}\leq\sqrt{N},  \quad t\leq0;\label{eq-P^c}\\
			&\|e^{kLt}\mathcal{P}^{s}\|_{\mathscr{L}(\mathbb{R}^{N},\mathbb{R}^{N})}\leq p(t)e^{-kt},\quad t\geq0,\label{eq-P^s}
		\end{align}
		where $\|\cdot\|_{\mathscr{L}(\mathbb{R}^{N},\mathbb{R}^{N})}$ denotes the operator norm  on $\mathscr{L}(\mathbb{R}^{N},\mathbb{R}^{N})$ and $p(t)$ is given in \eqref{ator-lem-pt}.    
    \end{lemma}
    \begin{proof}
		We only prove \eqref{eq-P^s}, as the proof of \eqref{eq-P^c} is straightforward. It follows from \eqref{ator-eqn-ekLt} and \eqref{ator-eqn-PcPs} that{\allowdisplaybreaks
        \fontsize{8}{9.6}\selectfont\begin{align*}
            \|e^{kLt}\mathcal{P}^{s}\|_{\mathscr{L}(\mathbb{R}^{N},\mathbb{R}^{N})}&\leq e^{-kt}\sqrt{\sum_{m=0}^{N-2}(N-m-1)\left[\frac{(kt)^{m}}{m!}\right]^{2}+\sum_{m=0}^{N-2}\bigg[\sum_{j=0}^{m}\frac{(kt)^{j}}{j!}\bigg]^{2}}\\
            &\leq e^{-kt}\sum_{m=0}^{N-2}\bigg(\sqrt{N-m-1}\frac{(kt)^{m}}{m!}+\sum_{j=0}^{m}\frac{(kt)^{j}}{j!}\bigg)\\
            &\leq e^{-kt}\sum_{m=0}^{N-2}\bigg((N-m-1)\frac{(kt)^{m}}{m!}+\sum_{j=0}^{m}\frac{(kt)^{j}}{j!}\bigg)\\
            &=e^{-kt}p(t)
        \end{align*}}for $t\geq 0$, which completes the proof.
    \end{proof}

    Before proving the existence of center-unstable manifold $W^{cu}$, we define  
    \begin{align*}
		\mathcal{C}_{\gamma}^{-}&=\left\{f\colon(-\infty,0]\to\mathbb{R}^{N}~\text{~is~continuous~and~}\sup_{t\leq0}\|e^{\gamma t}f(t)\|<\infty\right\}
    \end{align*}
    and
    \begin{align*}
        \mathcal{C}_{\gamma}^{+}&=\left\{f\colon[0,+\infty)\to\mathbb{R}^{N}\text{~is~continuous~and~}\sup_{t\geq0}\| e^{\gamma t}f(t)\|<\infty\right\}
    \end{align*}
    for $0<\gamma<k$. Both of them are Banach spaces equipped with the norms 
    \begin{align*}
		\|f\|_{\mathcal{C}_{\gamma}^{-}}=\sup_{t\leq0}\|e^{\gamma t}f(t)\|~~\text{and}~~ \|f\|_{\mathcal{C}_{\gamma}^{+}}=\sup_{t\geq0}\|e^{\gamma t}f(t)\|,
    \end{align*}
    respectively. Denote by $W^{cu}(\omega)$ the set of initial data for system \eqref{prly-eqn-phi} whose solution is controlled by $e^{-\gamma t}$ with $t\leq 0$, that is,
    \begin{align*}
		W^{cu}(\omega)=\left\{\phi_{0}\in\mathbb{R}^{N}\mid\phi(\cdot,\phi_{0},\omega)\in \mathcal{C}_{\gamma}^{-}\right\},\quad \omega\in\Omega.
    \end{align*}

    We then present an equivalent characterization of $W^{cu}(\omega)$ for $\omega\in\Omega$ (see also \cite{DLB2004} for more details).
    \begin{lemma}\label{mfld-lem-2}
		$\phi_{0}\in W^{cu}(\omega)$ if and only if there exists $\bar{\phi}\colon(-\infty,0]\to\mathbb{R}^{N}$ satisfying $\bar{\phi}(0)=\phi_{0}$, $\bar{\phi}\in \mathcal{C}_{\gamma}^{-}$ and
        \begin{equation}\label{mfld-lem1-eqn1}
            \begin{cases}
                \begin{aligned}
				\mathcal{P}^{c}\bar{\phi}(t)=&e^{kLt}\xi+\int_{0}^{t}e^{kL(t-u)}\mathcal{P}^{c}F\left(\bar{\phi}(u),\theta_{u}\omega\right)\,\mathrm{d}u,\\
				\mathcal{P}^{s}\bar{\phi}(t)=&\int_{-\infty}^{t}e^{kL(t-u)}\mathcal{P}^{s}F\left(\bar{\phi}(u),\theta_{u}\omega\right)\,\mathrm{d}u
			\end{aligned}
            \end{cases}
		\end{equation}
		for $t\leq0$, where $\xi=\mathcal{P}^{c}\phi_{0}$.
    \end{lemma}
    \begin{proof}
		We first let $\phi_{0}\in W^{cu}(\omega)$ and prove the existence of $\bar{\phi}$ satisfying \eqref{mfld-lem1-eqn1}. By using the variation of constants formula, the solution $\phi(t,\phi_0,\omega)$ of system \eqref{prly-eqn-phi} with initial data $\phi_{0}$ can be written as 
		\begin{equation}\label{mfld-pflem1-eqn1}
			\begin{aligned}
				\mathcal{P}^{c}\phi(t,\phi_{0},\omega)&=e^{kL(t-\tau)}\mathcal{P}^{c}\phi(\tau,\phi_{0},\omega)\\
				&\quad+\int_{\tau}^{t}e^{kL(t-u)}\mathcal{P}^{c}F(\phi(u,\phi_{0},\omega),\theta_{u}\omega)\,\mathrm{d}u,~~\text{for}~~\tau\geq t,
			\end{aligned}
		\end{equation}
		and
		\begin{equation}\label{mfld-pflem1-eqn2}
			\begin{aligned}
				\mathcal{P}^{s}\phi(t,\phi_{0},\omega)&=e^{kL(t-\tau)}\mathcal{P}^{s}\phi(\tau,\phi_{0},\omega)\\
				&\quad+\int_{\tau}^{t}e^{kL(t-u)}\mathcal{P}^{s}F(\phi(u,\phi_{0},\omega),\theta_{u}\omega)\,\mathrm{d}u,~~\text{for}~~ \tau\leq t.
			\end{aligned}
		\end{equation}
		Since
        \begin{align*}
            \left\| e^{k L  (t-\tau)}\mathcal{P}^{s}\phi(\tau,\phi_{0},\omega)\right\|&\leq p(t-\tau)e^{-k(t-\tau)}\|\phi(\tau,\phi_{0},\omega)\|\\
            &\leq e^{-k t}e^{(k-\gamma)\tau}p(t-\tau)\Vert\phi\Vert_{\mathcal{C}_{\gamma}^{-}}\to0~~(\text{as}~~\tau\to-\infty),
        \end{align*}
		\eqref{mfld-pflem1-eqn2} reduces to
		\begin{equation}\label{mfld-pflem1-eqnstable}
			\mathcal{P}^{s}\phi(t,\phi_{0},\omega)=\int_{-\infty}^{t}e^{kL(t-u)}\mathcal{P}^{s}F(\phi(u,\phi_{0},\omega),\theta_{u}\omega)\,\mathrm{d}u.
		\end{equation}
		Let $\bar{\phi}(t)=\phi(t,\phi_{0},\omega)$. By setting $\tau=0$ in \eqref{mfld-pflem1-eqn1} and using \eqref{mfld-pflem1-eqnstable}, we obtain that $\bar{\phi}$ satisfies \eqref{mfld-lem1-eqn1}, $\bar{\phi}(0)=\phi_0$, and $\bar{\phi}\in \mathcal{C}_{\gamma}^{-}$.
	
		Conversely, a direct computation shows that if $\bar{\phi}\in\mathcal{C}_{\gamma}^{-}$ satisfies \eqref{mfld-lem1-eqn1} and $\bar{\phi}(0)=\phi_{0}$, then $\phi_{0}\in W^{cu}(\omega)$. Thus, we have completed the proof.
	\end{proof}
	
    Now, we are ready to prove  Proposition \ref{mfld-prop-mfld} that  $W^{cu}(\omega)\neq \emptyset$ for $\omega\in\Omega$, is a one-dimensional and $C^{1}$-smooth manifold for coupling coefficient $k$ sufficiently large.
    \begin{proof}[Proof of \rm Proposition \ref{mfld-prop-mfld}]
		We first prove the existence and one-dimensionality for $W^{cu}$. It follows from Lemma \ref{mfld-lem-2} that   $W^{cu}(\omega)\neq \emptyset$ is equivalent to the existence of the solution of \eqref{mfld-lem1-eqn1} for $\omega\in\Omega$.

		For $\xi\in E^{c},\omega\in\Omega$, define the Lyapunov-Perron operator $\mathcal{L}_{\xi}\colon C^{-}_{\gamma }\to C^{-}_{\gamma};\bar{\phi}\mapsto\mathcal{L}_{\xi}\Big(\bar{\phi}\Big)$ by
		\begin{equation}\label{ator-eqn-PcLxi}
		  \begin{aligned}				
                \mathcal{P}^{c}\mathcal{L}_{\xi}\Big(\bar{\phi}\Big)(t)=e^{kLt}\xi + \int_{0}^t e^{kL(t-u)}\mathcal{P}^{c} F\Big(\bar{\phi}(u),\theta_u \omega\Big)\,\mathrm{d}u,~~\text{for}~~t\leq0,
		  \end{aligned}
		\end{equation}
        and
        \begin{equation}\label{ator-eqn-PsLxi}
		  \begin{aligned}	
                \mathcal{P}^{s}\mathcal{L}_{\xi}\Big(\bar{\phi}\Big)(t)&=\int_{-\infty}^t e^{kL(t-u)} \mathcal{P}^{s} F\Big(\bar{\phi}(u),\theta_u\omega\Big)\, \mathrm{d}u,~~\text{for}~~ t\leq0.
		  \end{aligned}
		\end{equation}	
        It is not hard to see from \eqref{ator-eqn-estimateF} and \eqref{eq-P^c}-\eqref{eq-P^s} that the operator $\mathcal{L}_{\xi}\colon\mathcal{C}^{-}_{\gamma}\to \mathcal{C}^{-}_{\gamma}$ is well defined.

		Furthermore, for any $\bar{\phi},\mathring{\phi}\in\mathcal{C}_{\gamma}^{-}$, a direct computation combining the Lipschitz continuity of $F(\cdot,\omega)$ with \eqref{eq-P^c}-\eqref{eq-P^s} yields that
        \begin{align}\label{mfld-pfprop-eqncenter}
			\left\|\mathcal{P}^{c}\mathcal{L}_{\xi}\Big(\bar{\phi}\Big)-\mathcal{P}^{c} \mathcal{L}_{\xi}\Big(\mathring{\phi}\Big)\right\|_{\mathcal{C}^{-}_{\gamma}}\leq\frac{\sqrt{N}}{\gamma}\cdot\operatorname{Lip}_{\phi}F\cdot\Big\|\bar{\phi}-\mathring{\phi}\Big\|_{\mathcal{C}^{-}_{\gamma}}
		\end{align}and
		{\allowdisplaybreaks\begin{equation}\label{mfld-pfprop-eqnstable}
	       \begin{aligned}
	           \left\|\mathcal{P}^{s}\mathcal{L}_{\xi}\Big(\bar{\phi}\Big)-\mathcal{P}^{s} \mathcal{L}_{\xi}\Big(\mathring{\phi}\Big)\right\|_{\mathcal{C}^{-}_{\gamma}}&\leq\operatorname{Lip}_{\phi}F\cdot\Big\|\bar{\phi}-\mathring{\phi}\Big\|_{\mathcal{C}^{-}_{\gamma}}\cdot\int_{-\infty}^{0}p(-u)e^{(k-\gamma)u}\,\mathrm{d}u\\
				&\leq\sum\limits_{m=1}^{N-1}\frac{2k^{m-1}(N-m)}{(k-\gamma)^{m}}\cdot\operatorname{Lip}_{\phi}F\cdot\Big\|\bar{\phi}-\mathring{\phi}\Big\|_{\mathcal{C}^{-}_{\gamma}},
	       \end{aligned}
		\end{equation}}where $p(u)=\sum_{m=0}^{N-2}2(N-m-1)\frac{(ku)^m}{m!}$ and we have utilized the fact that 
		\begin{align*}
			\int_{-\infty}^{0}(-u)^{m}e^{(k-\gamma)u}\,\mathrm{d}u=\frac{m!}{(k-\gamma)^{m+1}},\quad m\geq0
		\end{align*}
        in the last inequality of \eqref{mfld-pfprop-eqnstable}. Together with \eqref{mfld-pfprop-eqncenter}-\eqref{mfld-pfprop-eqnstable}, we obtain
		{\allowdisplaybreaks\begin{align*}
			\left\|\mathcal{L}_{\xi}\Big(\bar{\phi}\Big)- \mathcal{L}_{\xi}\Big(\mathring{\phi}\Big)\right\|_{\mathcal{C}^{-}_{\gamma}}
			\leq\left[\frac{\sqrt{N}}{\gamma}+\sum\limits_{m=1}^{N-1}\frac{2k^{m-1}(N-m)}{(k-\gamma)^{m}}\right]\cdot\operatorname{Lip}_{\phi}F\cdot\Big\|\bar{\phi}-\mathring{\phi}\Big\|_{\mathcal{C}^{-}_{\gamma}}.
		\end{align*}}Let $k_{0}$ be such that, for some $0<\gamma<k_{0}$,
        \begin{align}\label{eq-gap}
			\left[\frac{\sqrt{N}}{\gamma}+\sum\limits_{m=1}^{N-1}\frac{2k^{m-1}_0(N-m)}{(k_0-\gamma)^{m}}\right]\cdot\operatorname{Lip}_{\phi}F <1.
		\end{align}
        Then the Lyapunov-Perron operator $\mathcal{L}_{\xi}\colon\mathcal{C}^{-}_{\gamma}\to \mathcal{C}^{-}_{\gamma}$ is a contraction. Consequently, by using the uniform contraction principle, one obtains that for each $\xi\in E^{c}$ and $\omega\in\Omega$, the operator $\mathcal{L}_{\xi}\colon\mathcal{C}^{-}_{\gamma}\to \mathcal{C}^{-}_{\gamma}$ has a unique fixed point in $\mathcal{C}^{-}_{\gamma}$, denoted by  $\phi^{-}(\cdot,\xi,\omega)$. Hence, Lemma \ref{mfld-lem-2} implies that
		\begin{equation}\label{ator-eqn-Wcu}
			W^{cu}(\omega)=\{\xi+h(\xi,\omega)\mid\xi\in E^{c}\}.
		\end{equation}
		Here $h(\xi,\omega)=\int_{-\infty}^{0}e^{-k L  u}\mathcal{P}^{s}F\Big(\phi^{-}(u,\xi,\omega),\theta_{u}\omega\Big)\,\mathrm{d}u$ is Lipschitz continuous with respect to $\xi$,  for each $\omega\in\Omega$, due to \eqref{eq-gap}. For the mapping $h$, it follows from \eqref{ator-eqn-estimateF} and \eqref{eq-P^s} that, for any $\xi\in E^{c}$ and $\omega\in\Omega$,
        \begin{equation}\label{mfld-pfprop1-eqnh1}
            \begin{aligned}
                \| h(\xi,\omega)\|\leq C_{2}(\omega)
                \triangleq\int_{-\infty}^{0}p(-u)e^{ku}\cdot\sqrt{N}\big[\alpha_{*}+C_{1}(\omega)(1+|u|)\big]\,\mathrm{d}u.
            \end{aligned}
		\end{equation}
		Combining \eqref{mfld-pfprop1-eqnh1} with \eqref{ator-pfprop1-eqnC1}, we further obtain
		\begin{equation}\label{mfld-pfprop1-eqnh2}
		  \vert C_{2}(\theta_{t}\omega)\vert\leq C_{2}(\omega)(1+\vert t\vert),\quad t\in\mathbb{R}.
		\end{equation}
        Hence, $W^{cu}\,\bmod\,2\pi\mathbf{e}$ is an $M$-valued tempered random set. Thus, we have obtained the existence of random center-unstable manifold $W^{cu}$ which is one-dimensional since $E^{c}$ is one-dimensional. The $C^{1}$-smoothness of $W^{cu}$ for $k\geq k_{0}$ is standard, we omit it here and refer the reader to \cite{DLB2004,ShenLuZhang20}.
    \end{proof}
	
    Next, we will show that unwrapping $\mathscr{A}$ of $\mathcal{A}$ from $M$ to $\mathbb{R}^{N}$ (see \eqref{dp-coro-eqnscrA}) satisfies $\mathscr{A}=W^{cu}$, when the coupling coefficient $k$ is sufficiently large.

    \begin{proposition}\label{do-prop-mfld=ator}
		Let $k_{0}>0$ be given by \eqref{eq-gap}. Then $\mathscr{A}=W^{cu}$ whenever $k\geq k_{0}$.
    \end{proposition}

    Before proving Proposition \ref{do-prop-mfld=ator}, we first show how it implies Theorem \ref{mainth-a}.\vspace{0.35\baselineskip}

    \textit{Proof of} \rm Theorem \ref{mainth-a}: By virtue of Lemma \ref{ator-lem-exist}, it suffices to show that $\mathcal{A}(\omega),\omega\in\Omega$, is a closed continuously differentiable curve on $M$. Due to $\mathscr{A}=W^{cu}$ for $k\ge k_{0}$ (in Proposition \ref{do-prop-mfld=ator}), and $W^{cu}(\omega)$ is continuously differentiable, we obtain that $\mathcal{A}(\omega)$ is a $C^{1}$-smooth curve on $M$. It remains to show $\mathcal{A}(\omega)$ is a closed curve. To this end, we first claim that
        \begin{align*}
            \phi^{-}(t,\xi+2\pi\mathbf{e},\omega)=\phi^{-}(t,\xi,\omega)+2\pi\mathbf{e},\quad \forall~t\leq 0,\,\xi\in E^{c},\,\omega\in\Omega,
        \end{align*}
        where $\phi^{-}(\cdot,\xi,\omega)$ is the fixed point of the Lyapunov-Perron operator $\mathcal{L}_{\xi}\colon\mathcal{C}_{\gamma}^{-}\to\mathcal{C}_{\gamma}^{-}$ defined in \eqref{ator-eqn-PcLxi}-\eqref{ator-eqn-PsLxi}. In fact, by following the procedures of the proof of Proposition \ref{mfld-prop-mfld}, one can check that for any $\xi\in E^{c}$, $\phi^{-}(\cdot,\xi,\omega)+2\pi\mathbf{e}$ is a fixed point of Lyapunov-Perron operator $\mathcal{L}_{\xi+2\pi\mathbf{e}}\colon\mathcal{C}^{-}_{\gamma}\to \mathcal{C}^{-}_{\gamma}$. The claim then follows from the uniqueness of the fixed point.

        Based on this claim, one can deduce from \eqref{ator-eqn-F} that 
        \begin{align*}
            F(\phi^{-}(t,\xi+2\pi\mathbf{e},\omega),\theta_{t}\omega)=F(\phi^{-}(t,\xi,\omega),\theta_{t}\omega), \quad \forall~t\leq 0,\,\xi\in E^{c},\,\omega\in\Omega.
		\end{align*}
        Together with the definition of $h$ in \eqref{ator-eqn-Wcu}, we obtain
        \begin{align*}
            h(\xi+2\pi\mathbf{e},\omega)=h(\xi,\omega),\quad\forall~\xi\in E^{c},\,\omega\in\Omega.
        \end{align*}
        This implies that, for any $\xi\in E^{c}$ and $\omega\in\Omega$,
        \begin{align*}
            h(\xi+2\pi\mathbf{e},\omega)+\xi+2\pi\mathbf{e}=h(\xi,\omega)+\xi+2\pi\mathbf{e}.
		\end{align*}
        So, by \eqref{ator-eqn-Wcu}, we have obtained that $W^{cu}(\omega)\,\bmod\,2\pi\mathbf{e}$, for each $\omega\in\Omega$, is a closed curve on $M$. Hence, $\mathcal{A}(\omega)$, for each $\omega\in\Omega$, is a closed $C^{1}$-smooth curve on $M$ for sufficiently large $k$. Thus we have completed the proof.\hfill $\square$
    \vspace{0.5\baselineskip} 
   
	We now turn to the proof of Proposition \ref{do-prop-mfld=ator}. We begin by establishing several useful lemmas. For any $\xi\in E^{c}$,  denote by $W^{cu}_{s}(\xi+h(\xi,\omega),\omega)$ the set of initial data whose forward orbits converge exponentially to $\phi(\cdot,\xi+h(\xi,\omega),\omega)$, that is,
    \begin{align*}
        W^{cu}_{s}(\xi+h(\xi,\omega),\omega)\triangleq\left\{\rho\in\mathbb{R}^{N}\mid\phi(\cdot,\rho,\omega)-\phi(\cdot,\xi+h(\xi,\omega),\omega)\in\mathcal{C}^{+}_{\gamma}\right\},\quad\omega\in\Omega.
    \end{align*}
    
    \begin{lemma}\label{do-lem-foliation} 
        Let $k_{0}>0$ be given by \eqref{eq-gap}. Then $W^{cu}$ possesses stable foliations, whose leafs coincide with $W^{cu}_s$ for $k\geq k_{0}$.
    \end{lemma}
	\begin{proof}        
		For any $\xi\in E^{c}$, $\eta\in E^{s}$ and $\omega\in\Omega$, define the Lyapunov-Perron  operator on $\mathcal{C}_{\gamma}^{+}$ as 
		$\mathcal{B}_{\eta}\colon\mathcal{C}^{+}_{\gamma}\to \mathcal{C}^{+}_{\gamma};\check{\phi}\mapsto\mathcal{B}_{\eta}\left(\check{\phi}\right)$, where
		\begin{align*}
			\mathcal{P}^{s}\mathcal{B}_{\eta}\Big(\check{\phi}\Big)(t)&=e^{kLt}\eta+\int_{0}^t e^{kL(t-u)} \mathcal{P}^{s} \Big[ F\Big( \check{\phi}(u)+\phi(u,\xi + h(\xi, \omega),\omega),\theta_u \omega\Big)\\
				&\quad-F(\phi(u,\xi+h(\xi, \omega),\omega),\theta_u \omega)\Big]\,\mathrm{d}u
            \end{align*}
            and
            \begin{align*}
                \mathcal{P}^{c}\mathcal{B}_{\eta}\Big(\check{\phi}\Big)(t)&=\int_{\infty}^t e^{kL(t-u)} \mathcal{P}^{c}\Big[ F\Big( \check{\phi}(u)+\phi(u,\xi + h(\xi, \omega),\omega),\theta_u \omega\Big)\\
				&\quad-F(\phi(u,\xi+h(\xi, \omega),\omega),\theta_u \omega) \Big]\,\mathrm{d}u
		\end{align*}
        for $t\geq0$. It is not hard to see from the Lipschitz continuity of $F(\cdot,\omega)$ and \eqref{eq-P^c}-\eqref{eq-P^s} that the operator $\mathcal{B}_\eta{\colon\mathcal{C}^{+}_{\gamma}\to \mathcal{C}^{+}_{\gamma}}$ is well defined. By following the similar arguments as in Lemma \ref{mfld-lem-2}, one can obtain $\rho\in W^{cu}_{s}(\xi+h(\xi,\omega),\omega)$ if and only if the solution of $\phi$ defined by system \eqref{prly-eqn-phi} with initial data $\rho$ is the fixed point of the operator $\mathcal{B}_\eta\colon\mathcal{C}^{+}_{\gamma}\to \mathcal{C}^{+}_{\gamma}$ with $\eta=\mathcal{P}^{s}\rho$. 
        
        We now prove the Lyapunov-Perron operator $\mathcal{B}_{\eta}\colon\mathcal{C}^{+}_{\gamma}\to \mathcal{C}^{+}_{\gamma}$ is a contraction. In fact, for $\check{\phi},\mathring{\phi}\in\mathcal{C}_{\gamma}^{+}$, it follows from the Lipschitz continuity of $F(\cdot,\omega)$ and \eqref{eq-P^c}-\eqref{eq-P^s} that
		\begin{align}\label{do-pflem1-eqncenter2}
			\left\|\mathcal{P}^{c}\mathcal{B}_{\eta}\Big(\check{\phi}\Big)-\mathcal{P}^{c} \mathcal{B}_{\eta}\Big(\mathring{\phi}\Big)\right\|_{\mathcal{C}^{+}_{\gamma}}\leq\frac{\sqrt{N}}{\gamma}\cdot\operatorname{Lip}_{\phi}F\cdot\left\|\check{\phi}-\mathring{\phi}\right\|_{\mathcal{C}^{+}_{\gamma}}
		\end{align}and 
	{\allowdisplaybreaks\begin{equation}\label{do-pflem1-eqnstable2}
		\begin{aligned}
            \left\|\mathcal{P}^{s}\mathcal{B}_{\eta}\Big(\check{\phi}\Big)-\mathcal{P}^{s} \mathcal{B}_{\eta}\Big(\mathring{\phi}\Big)\right\|_{\mathcal{C}^{+}_{\gamma}}&\leq\operatorname{Lip}_{\phi}F\cdot\left\|\check{\phi}-\mathring{\phi}\right\|_{\mathcal{C}^{+}_{\gamma}}\cdot\int_{-\infty}^{0}p(-u)e^{(k-\gamma)u}\mathrm{d}u\\
				&\leq\sum\limits_{m=1}^{N-1}\frac{2k^{m-1}(N-m)}{(k-\gamma)^{m}}\cdot\operatorname{Lip}_{\phi}F\cdot\Big\|\check{\phi}-\mathring{\phi}\Big\|_{\mathcal{C}^{+}_{\gamma}}.
		\end{aligned}
		\end{equation}}By \eqref{do-pflem1-eqncenter2}-\eqref{do-pflem1-eqnstable2}, one has
		\begin{align*}
			\left\|\mathcal{B}_{\eta}\Big(\check{\phi}\Big)- \mathcal{B}_{\eta}\Big(\mathring{\phi}\Big)\right\|_{\mathcal{C}^{+}_{\gamma}}\leq\left[\frac{\sqrt{N}}{\gamma}+\sum\limits_{m=1}^{N-1}\frac{2k^{m-1}(N-m)}{(k-\gamma)^{m}}\right]\cdot\operatorname{Lip}_{\phi}F \cdot\Big\|\check{\phi}-\mathring{\phi}\Big\|_{\mathcal{C}^{+}_{\gamma}}.
		\end{align*}
        Combining with \eqref{eq-gap}, it entails that $\mathcal{B}_{\eta}\colon\mathcal{C}^{+}_{\gamma}\to \mathcal{C}^{+}_{\gamma}$ is a contraction, for $k\ge k_{0}$. So, by using the uniform contraction principle, we obtain that, for each $\omega\in\Omega,\xi\in E^{c}$ and $\eta\in E^{s}$, the operator $\mathcal{B}_{\eta}\colon\mathcal{C}^{+}_{\gamma}\to \mathcal{C}^{+}_{\gamma}$ has a unique fixed point in $\mathcal{C}^{+}_{\gamma}$, denoted by $\phi^{+}(\cdot,\xi,\eta,\omega)$. 
		
		Now, for $\omega\in\Omega,\xi\in E^{c}$, it follows from the definition of $W^{cu}_{s}(\xi+h(\xi,\omega),\omega)$ that
		\begin{equation}\label{ator-lem2-Wcus}
			\begin{aligned}
				W^{cu}_{s}(\xi+h(\xi,\omega),\omega)&=\left\{\rho\in\mathbb{R}^{N}\;\middle|\;\phi(\cdot,\rho,\omega)-\phi(\cdot,\xi+h(\xi,\omega),\omega)\in\mathcal{C}^{+}_{\gamma}\right\}\\
				&=\left\{\xi+h(\xi,\omega)+\phi^{+}(0,\xi,\eta,\omega)\;\middle|\;\eta\in E^{s}\right\}\\
				&=\left\{\eta+h(\xi,\omega)+\tilde{h}(\xi,\eta,\omega)\;\middle|\;\eta\in E^{s}\right\},
			\end{aligned}
		\end{equation}
		where 
		\begin{equation}\label{eqn-tildeh}
            \begin{aligned}
                \tilde{h}(\xi,\eta,\omega)&=\xi+\int_{\infty}^{0}e^{-kLu} \mathcal{P}^{c} \Big[F\Big(\phi^{+}(u,\xi,\eta,\omega)+\phi(u,\xi + h(\xi, \omega),\omega),\theta_u \omega\Big)\\
				&\quad- F( \phi(u,  \xi + h(\xi, \omega),\omega),\theta_u \omega) \Big]\,\mathrm{d}u.
            \end{aligned}
		\end{equation}
        Thus, we have completed the proof of the lemma.
    \end{proof}
	
    Next, we give a technical estimate of the fixed point $\phi^{+}(\cdot,\xi,\eta,\omega)$ of the operator $\mathcal{B}_{\eta}\colon\mathcal{C}^{+}_{\gamma}\to \mathcal{C}^{+}_{\gamma}$, which is useful in the proof of Proposition \ref{do-prop-mfld=ator}.
    \begin{lemma}
		For any $\omega\in\Omega,\xi\in E^{c}$ and $\eta\in E^{s}$, $\phi^{+}(\cdot,\xi,\eta,\omega)$ satisfies 
		\begin{equation}\label{do-lem2-eqn1}
			\left\|\phi^{+}(\cdot,\xi,\eta,\omega)\right\|_{\mathcal{C}_{\gamma}^{+}}\leq \frac{C_{3}}{1-\left[\frac{\sqrt{N}}{\gamma}+\sum\limits_{m=1}^{N-1}\frac{2k^{m-1}(N-m)}{(k-\gamma)^{m}}\right]\cdot\operatorname{Lip}_{\phi}F } \|\eta\|,
        \end{equation}
		for $k\ge k_{0}$ and $\gamma$ as given by \eqref{eq-gap}, where $C_{3}=\sum\limits_{m=0}^{N-2}\frac{2(N-m-1)(km)^m}{m!(k-\gamma)^{m}e^{m}}$.
    \end{lemma}
    \begin{proof}
		Let $k\geq k_{0}$ and $\gamma$ be as given by \eqref{eq-gap}. We fix $\omega\in\Omega,\xi\in E^{c}$ and $\eta\in E^{s}$, and first show that
		\begin{equation}\label{do-lemma-hatpsi}
        {\fontsize{8}{9.6}\selectfont\begin{aligned}
			\Big\|\phi^{+}(\cdot,\xi,\eta,\omega)-\phi^{+}(\cdot,\xi,0, \omega)\Big\|_{\mathcal{C}^{+}_{\gamma}}\leq \frac{C_{3}}{1-\left[\frac{\sqrt{N}}{\gamma}+\sum\limits_{m=1}^{N-1}\frac{2k^{m-1}(N-m)}{(k-\gamma)^{m}}\right]\cdot\operatorname{Lip}_{\phi}F } \|\eta\|.
            \end{aligned}}
		\end{equation}
        The Lipschitz continuity of $F(\cdot,\omega)$, together with \eqref{eq-P^c}-\eqref{eq-P^s}, leads to
		\begin{equation}\label{do-pflem2-eqn1}
			{\fontsize{8}{9.6}\selectfont\begin{aligned}
			&\left\|\mathcal{B}_{\eta}\Big(\phi^{+}(\cdot,\xi,\eta,\omega)\Big) - \mathcal{B}_{\eta}\Big(\phi^{+}(\cdot,\xi,0,\omega)\Big)\right\|_{\mathcal{C}^{+}_{\gamma}} \\
				&\quad\leq\left[\frac{\sqrt{N}}{\gamma}+\sum\limits_{m=1}^{N-1}\frac{2k^{m-1}(N-m)}{(k-\gamma)^{m}}\right]\cdot\operatorname{Lip}_{\phi}F\cdot\Big\|\phi^{+}(\cdot,\xi,\eta,\omega)-\phi^{+}(\cdot,\xi,0, \omega)\Big\|_{\mathcal{C}^{+}_{\gamma}}
		\end{aligned}}
        \end{equation}and 
        {\allowdisplaybreaks\begin{equation}\label{do-pflem2-eqn2}
            \begin{aligned}
                \left\|\mathcal{B}_{\eta}\Big(\phi^{+}(\cdot,\xi,0,\omega)\Big)- \mathcal{B}_{0}\Big(\phi^{+}(\cdot,\xi,0, \omega)\Big)\right\|_{\mathcal{C}^{+}_{\gamma}}
                \leq\sup_{t\geq0}p(t)e^{-(k-\gamma)t}\|\eta\| =C_{3}\|\eta\|,
            \end{aligned}    
        \end{equation}}where $p(t)=\sum_{m=0}^{N-2}2(N-m-1)\frac{(kt)^m}{m!}$ and we  used the equality
		\begin{align*}
		  \sup_{t\geq0}t^{m}e^{-(k-\gamma)t}=\frac{m^m}{(k-\gamma)^{m}e^{m}},\quad (m\geq0)
		\end{align*} 
        in the equality of \eqref{do-pflem2-eqn2}.  Based on \eqref{do-pflem2-eqn1}-\eqref{do-pflem2-eqn2} and the fact that $\phi^{+}(\cdot,\xi,\eta,\omega)$ is the fixed point of the operator $\mathcal{B}_{\eta}\colon\mathcal{C}^{+}_{\gamma}\to \mathcal{C}^{+}_{\gamma}$, we obtain \eqref{do-lemma-hatpsi}.

        Now, we prove \eqref{do-lem2-eqn1}. This can be done by \eqref{do-lemma-hatpsi} and $\phi^{+}(t,\xi,0,\omega)=0$ for $t\geq0$. Thus, we have completed the proof. 
    \end{proof}
	
	Now, we are ready to prove Proposition \ref{do-prop-mfld=ator}:\vspace{0.35\baselineskip}

    \textit{Proof of} \rm Proposition \ref{do-prop-mfld=ator}: Let $k\geq k_{0}$ and $\gamma$ be as given by \eqref{eq-gap}. Fix any $\omega\in\Omega$. We first claim that $W^{cu}$ attracts any $\mathbb{R}^{N}$-valued random set $D$ satisfying $D\,\bmod\,2\pi\mathbf{e}\in\mathcal{D}$, where $\mathcal{D}$ is a family of $M$-valued tempered random sets.
		
		In fact, for every $t\geq0$ and  $\phi_{0}\in D(\theta_{-t}\omega)$, one can find some $\xi_{0}\in E^{c}$ such that
		\begin{equation*}
			\phi_{0}\in W^{cu}_s(\xi_0 + h(\xi_{0}, \theta_{-t} \omega), \theta_{-t} \omega),
		\end{equation*}
        where $W^{cu}_{s}$ is defined in \eqref{ator-lem2-Wcus}. Then, we write 
		\begin{equation}\label{do-pfprop1-eqn1}
			\phi_0 = \eta_0 + h(\xi_0, \theta_{-t} \omega) + \tilde{h}(\xi_0, \eta_0, \theta_{-t} \omega)\quad \text{with}\quad \eta_0\triangleq\mathcal{P}^{s}\phi_0-h(\xi_0,\theta_{-t}\omega),
		\end{equation}
		where $\tilde{h}$ is defined in \eqref{eqn-tildeh}. So,
    \begin{equation}\label{do-pfprop1-eqn2}
               {\fontsize{8}{9.6}\selectfont
        \begin{aligned}
			&\|\phi(t,\phi_0,\theta_{-t} \omega) - \phi(t,\xi_0 + h(\xi_0,\theta_{-t} \omega),\theta_{-t}\omega)\|\\
			=\quad\quad&\Big\|\phi^{+}(t, \xi_0, \eta_0, \theta_{-t}\omega)\Big\|\\
			\stackrel{\eqref{do-lem2-eqn1}+\eqref{do-pfprop1-eqn1}}{\leq}& \frac{C_{3} e^{-\gamma t}}{1-\left[\frac{\sqrt{N}}{\gamma}+\sum\limits_{m=1}^{N-1}\frac{2k^{m-1}(N-m)}{(k-\gamma)^{m}}\right]\cdot\operatorname{Lip}_{\phi}F }\big[\| \mathcal{P}^{s}\phi_0\|+\| h(\xi_0,\theta_{-t}\omega)\|\big]\\
			\stackrel{\eqref{mfld-pfprop1-eqnh1}+\eqref{mfld-pfprop1-eqnh2}}{\leq}&\frac{C_{3} e^{-\gamma t}}{1-\left[\frac{\sqrt{N}}{\gamma}+\sum\limits_{m=1}^{N-1}\frac{2k^{m-1}(N-m)}{(k-\gamma)^{m}}\right]\cdot\operatorname{Lip}_{\phi}F }\big[\| \mathcal{P}^{s}\phi_0\|+ C_{2}(\omega)(1+\vert t \vert)\big],
    \end{aligned}}
    \end{equation}
   where the equality follows from that $\phi^{+}(t,\xi_{0},\eta_{0},\omega)+\phi(t,\xi_{0}+h(\xi_{0},\omega),\omega)$ is the solution of $\phi$ defined in system \eqref{prly-eqn-phi} with initial data $\phi_{0}$.
		
		Therefore, for any $t\geq 0$, 
        {\allowdisplaybreaks\begin{align*}
            &\operatorname{dist}\big(\phi(t,D(\theta_{-t}\omega), \theta_{-t} \omega), W^{cu}(\omega)\big)\\
            \leq\;\,\,&\sup_{\phi_0 \in D(\theta_{-t}\omega)} \|\phi(t,\phi_0,\theta_{-t}\omega)-\phi(t,\xi_0 + h(\xi_0, \theta_{-t} \omega),\theta_{-t} \omega)\|\\
            \stackrel{\eqref{do-pfprop1-eqn2}}{\leq}&\sup_{\phi_0 \in D(\theta_{-t}\omega)}\frac{C_{3} e^{-\gamma t}}{1-\left[\frac{\sqrt{N}}{\gamma}+\sum\limits_{m=1}^{N-1}\frac{2k^{m-1}(N-m)}{(k-\gamma)^{m}}\right]\cdot\operatorname{Lip}_{\phi}F }\big[\| \mathcal{P}^{s}\phi_0\|+ C_{2}(\omega)(1+\vert t \vert)\big],
        \end{align*}}where we have used the invariance of $W^{cu}$ for the first  inequality. Thus, for $k\geq k_{0}$, one obtains
		\begin{equation*}
			\operatorname{dist}\big(\phi(t,D(\theta_{-t}\omega),\theta_{-t}\omega), W^{cu}(\omega)\big)\to0,~~\text{as}~~  t\to\infty.
		\end{equation*}
		So, we have proved the claim. Choose the $\mathbb{R}^{N}$-valued random set $D=\mathscr{A}$ (because $\mathcal{A}\in\mathcal{D}$). Then, it follows from the claim that 
        \begin{align*}
            \operatorname{dist}\big(\mathscr{A}(\omega),W^{cu}(\omega)\big)=\operatorname{dist}\big(\phi(t,\mathscr{A}(\theta_{-t}\omega),\theta_{-t}\omega),W^{cu}(\omega)\big)\to0,~~\text{as}~~t\to\infty.
        \end{align*}
        This implies $\mathscr{A}(\omega)\subset W^{cu}(\omega)$. 
        
        On the other hand, it follows from \eqref{mfld-pfprop1-eqnh1} and \eqref{mfld-pfprop1-eqnh2} that $\hat{W}^{cu}\triangleq W^{cu}\,\bmod\, 2\pi\mathbf{e}\in \mathcal{D}$. Recall that
        \begin{align*}
            \operatorname{dist}\big(\phi(t,W^{cu}(\theta_{-t}\omega),\theta_{-t}\omega),\mathscr{A}(\omega)\big)=\operatorname{dist}\big(\Phi(t,\hat{W}^{cu}(\theta_{-t}\omega),\theta_{-t}\omega),\mathcal{A}(\omega)\big).
        \end{align*}
       Hence, together with Lemma \ref{ator-lem-exist} and the invariance of $W^{cu}$, one obtains
       \begin{align*}
           \operatorname{dist}\big(W^{cu}(\omega),\mathscr{A}(\omega)\big)=\operatorname{dist}\big(\Phi(t,\hat{W}^{cu}(\theta_{-t}\omega),\theta_{-t}\omega),\mathcal{A}(\omega)\big)\to0,~~\text{as}~~t\to\infty.
       \end{align*} 
       This yields that $W^{cu}(\omega)\subset\mathscr{A}(\omega)$. Consequently, $\mathscr{A}=W^{cu}$ for $k\geq k_{0}$. Thus, we have completed the proof. \hfill $\square$

    \section{Numerical Simulations}\label{section-simulation}
        In this section, we elucidate the dynamical order in stochastic system \eqref{intro-eqn-psi} for $N=2,3$ from the perspective of numerical simulations. More precisely, a total order structure with respect to standard order ``$\leq$" in $R^{N}$ ($N=2,3$) and an intuitive, simple asymptotic one-dimensional structure in $M$ can be seen.

        Firstly, for $N=2$, we randomly select $500$ initial states from 
        $[0,2\pi]^{2}$ (cf. red points in Figure \ref{sim-fig-psi2}) with $k=3$, $(\alpha_{1},\alpha_{2})=(1,2)$, $(\epsilon_{1},\epsilon_{2})=(0.2,0.3)$ for stochastic system \eqref{intro-eqn-psi}. Blue points in Figure \ref{sim-fig-psi2}, trajectories from random initial states at the $100$th iteration reveal a total order structure with respect to the standard order ``$\le$" in $\mathbb{R}^{2}$. Figure \ref{sim-fig-tildepsi2} is obtained by wrapping the phase space $\mathbb{R}^{2}$ (of Figure \ref{sim-fig-psi2}) along the vector $(2\pi,2\pi)^{\top}$ to yield $M$ (diffeomorphic to $\mathbb{S}^{1}\times\mathbb{R}^{1}$) and viewing the red and blue points from Figure \ref{sim-fig-psi2} in $M$. In Figure \ref{sim-fig-tildepsi2}, we view $\mathbb{S}^{1}\times\mathbb{R}^{1}$; and moreover, the red points (initial states) are readily attracted to the simple asymptotic one-dimensional structure formed by the blue points, which themselves form a closed curve.
   
        \begin{figure}[H]
		\begin{minipage}[t]{0.45\textwidth}
			\centering
			\includegraphics[width=1.0\textwidth]{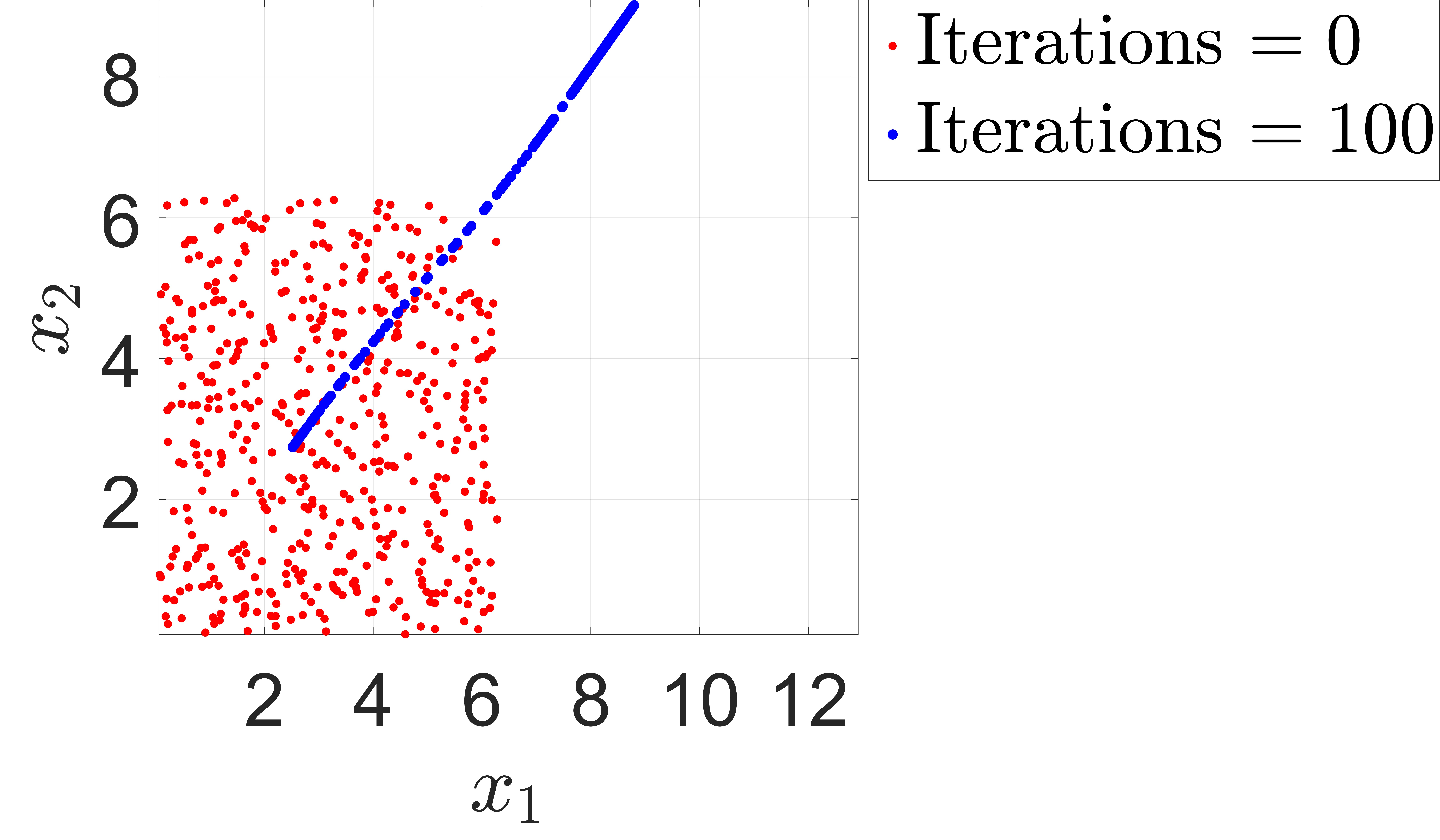}
                \caption{A total order structure in $\mathbb{R}^{2}$.}
			\label{sim-fig-psi2}
		\end{minipage}
		\hfill
		\begin{minipage}[t]{0.45\textwidth}
			\centering
			\includegraphics[width=1.0\textwidth]{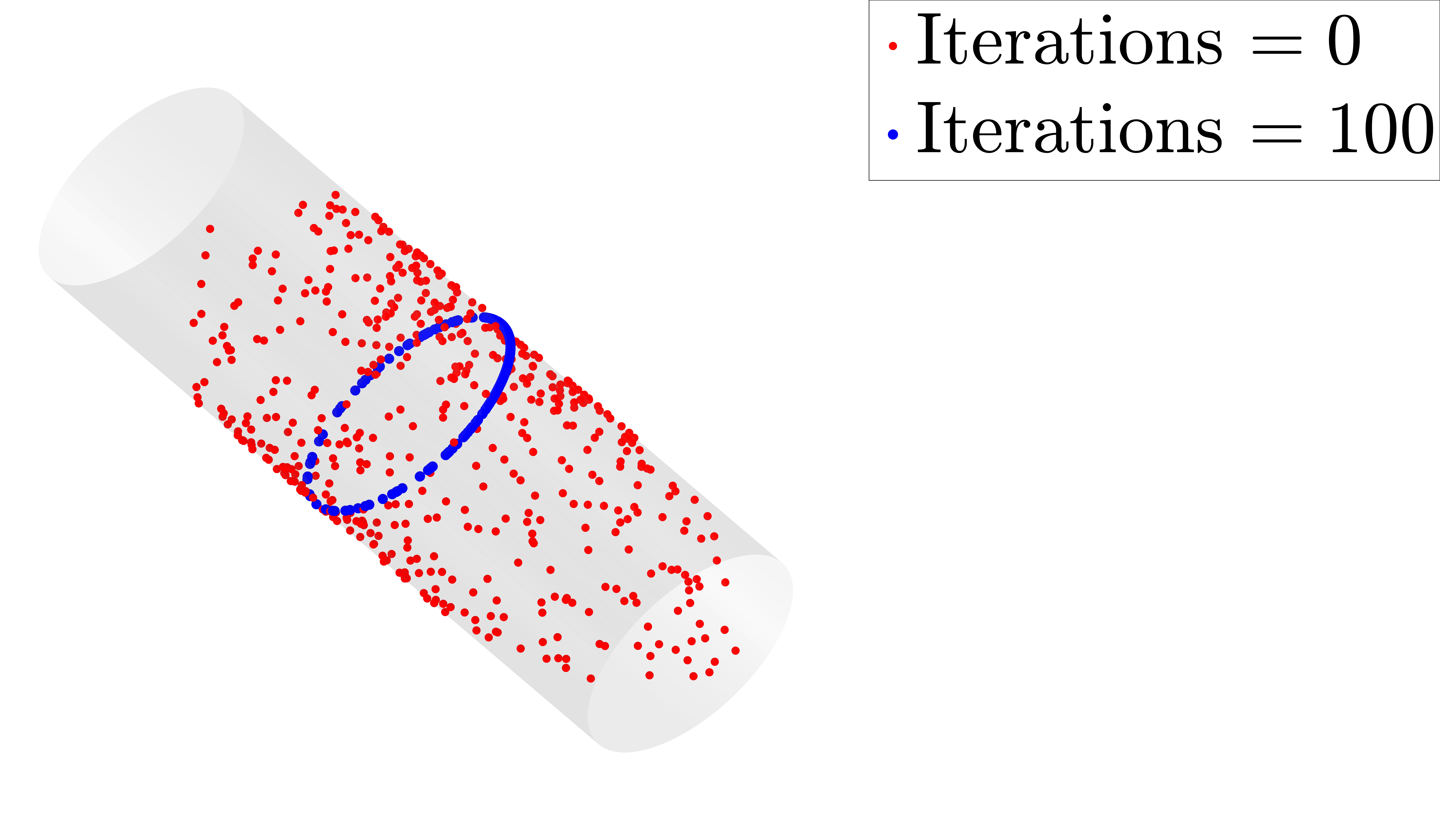}
			\caption{Wrapping along $(2\pi,2\pi)^{\top}$.}
			\label{sim-fig-tildepsi2}
		\end{minipage}
	\end{figure}
    
        Next, for $N=3$, we randomly select $500$ initial states from $[0,2\pi]^{3}$ (cf. red points in Figure \ref{sim-fig-psi3}) with $k=8$, $(\alpha_{1},\alpha_{2},\alpha_{3})=(1,2,3)$, $(\epsilon_{1},\epsilon_{2},\epsilon_3)=(0.1,0.2,0.3)$ for stochastic system \eqref{intro-eqn-psi}. Also, a total order structure with respect to the standard order ``$\le$" in $\mathbb{R}^{3}$ is revealed by the blue points in Figure \ref{sim-fig-psi3}.

         By wrapping the phase space $\mathbb{R}^{3}$ (of Figure \ref{sim-fig-psi3})  along the vector $(2\pi,2\pi,2\pi)^{\top}$, one can obtain $M$ (diffeomorphic to $\mathbb{S}^{1}\times\mathbb{R}^{2}$) and
        see the dynamics of stochastic system \eqref{intro-eqn-psi} in $M$. Note that $\mathbb{S}^{1}\times\mathbb{R}^{2}$ can be embedded in $\mathbb{R}^{4}$. Therefore, we can wrap the points in Figure \ref{sim-fig-psi3} along the vector $(2\pi,2\pi,2\pi)^{\top}$ and plot them in $\mathbb{R}^{4}$. More precisely, define the map $\zeta:\mathbb{R}^{3}\to \mathbb{S}^{1}\times\mathbb{R}^{2}\subset\mathbb{R}^{4};\,x\triangleq(x_{1},x_{2},x_{3})\mapsto \zeta(x)$ by
        {\small\begin{align*}
            \zeta(x)=\left(\cos\left(\frac{x_1+x_2+x_3}{3}\right),\,\sin\left(\frac{x_1+x_2+x_3}{3}\right),\,\frac{x_1-x_2}{\sqrt{2}},\,\frac{x_1+x_2-2x_3}{\sqrt{6}}\right).
        \end{align*}}For any point $y=(y_{1},y_{2},y_{3},y_{4})\in \mathbb{R}^4$, we orthogonally project it onto the coordinate hyperplanes $\{y\in\mathbb{R}^4\mid y_{4}=0\}$ and $\{y\in\mathbb{R}^4\mid y_{3}=0\}$ as shown in Figure \ref{sim-fig-tildepsi3-123} and Figure \ref{sim-fig-tildepsi3-124}, respectively. In both figures, the red points (initial states) are readily attracted to the simple asymptotic one-dimensional structure formed by the blue points. Note that the blue points in Figure \ref{sim-fig-tildepsi3-123} share the same third coordinate and the blue points in Figure \ref{sim-fig-tildepsi3-124} share the same fourth coordinate. Hence,  the blue points from Figure \ref{sim-fig-psi3} form a closed curve in view of $M$.
        \begin{figure}[H]
            \centering
            \includegraphics[width=0.6\linewidth]{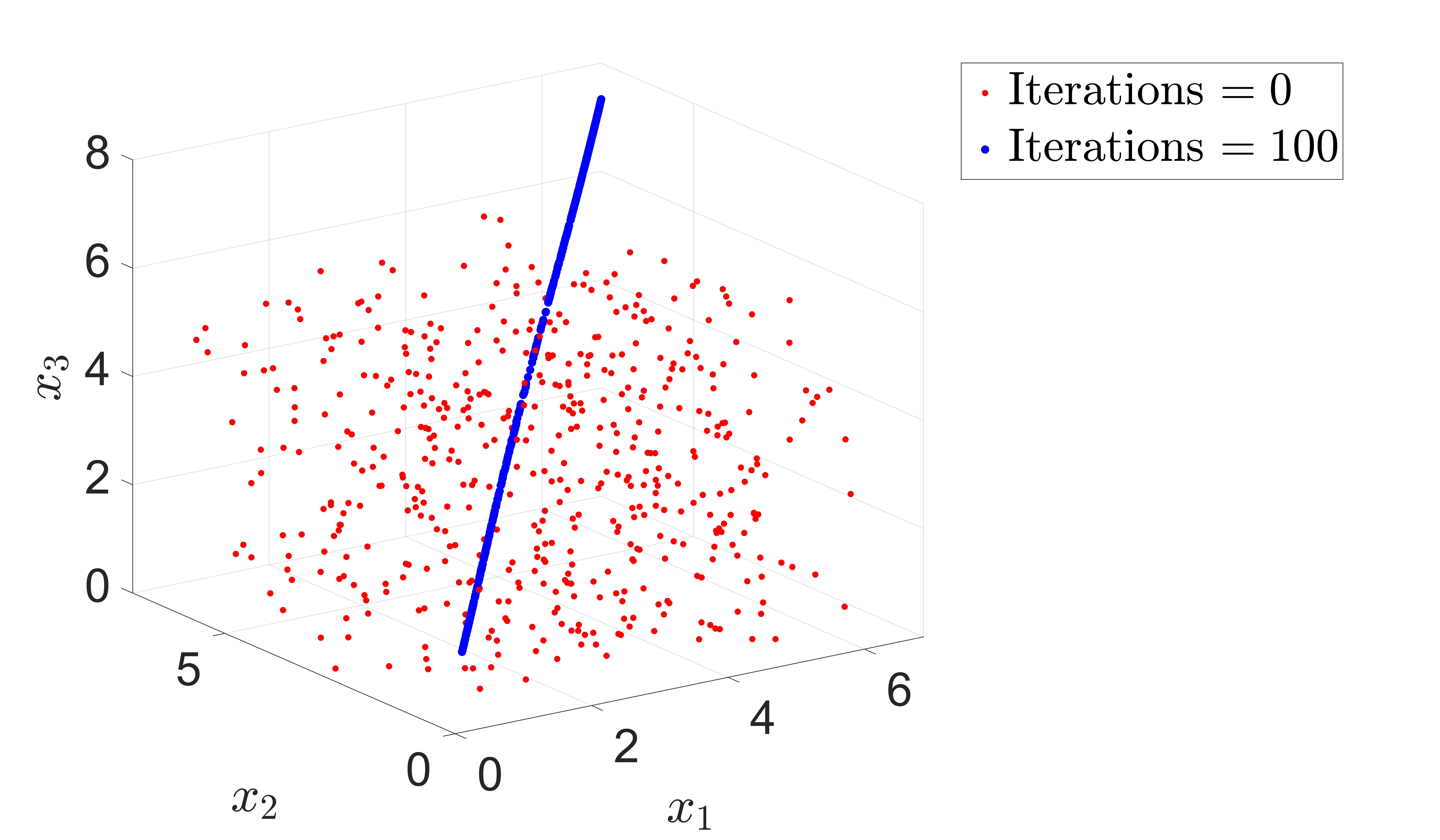}
            \caption{A total order structure in $\mathbb{R}^{3}$}
            \label{sim-fig-psi3}
        \end{figure}

        \begin{figure}[H]
            \begin{minipage}[t]{0.45\textwidth}
                \centering
                \includegraphics[width=1.0\textwidth]{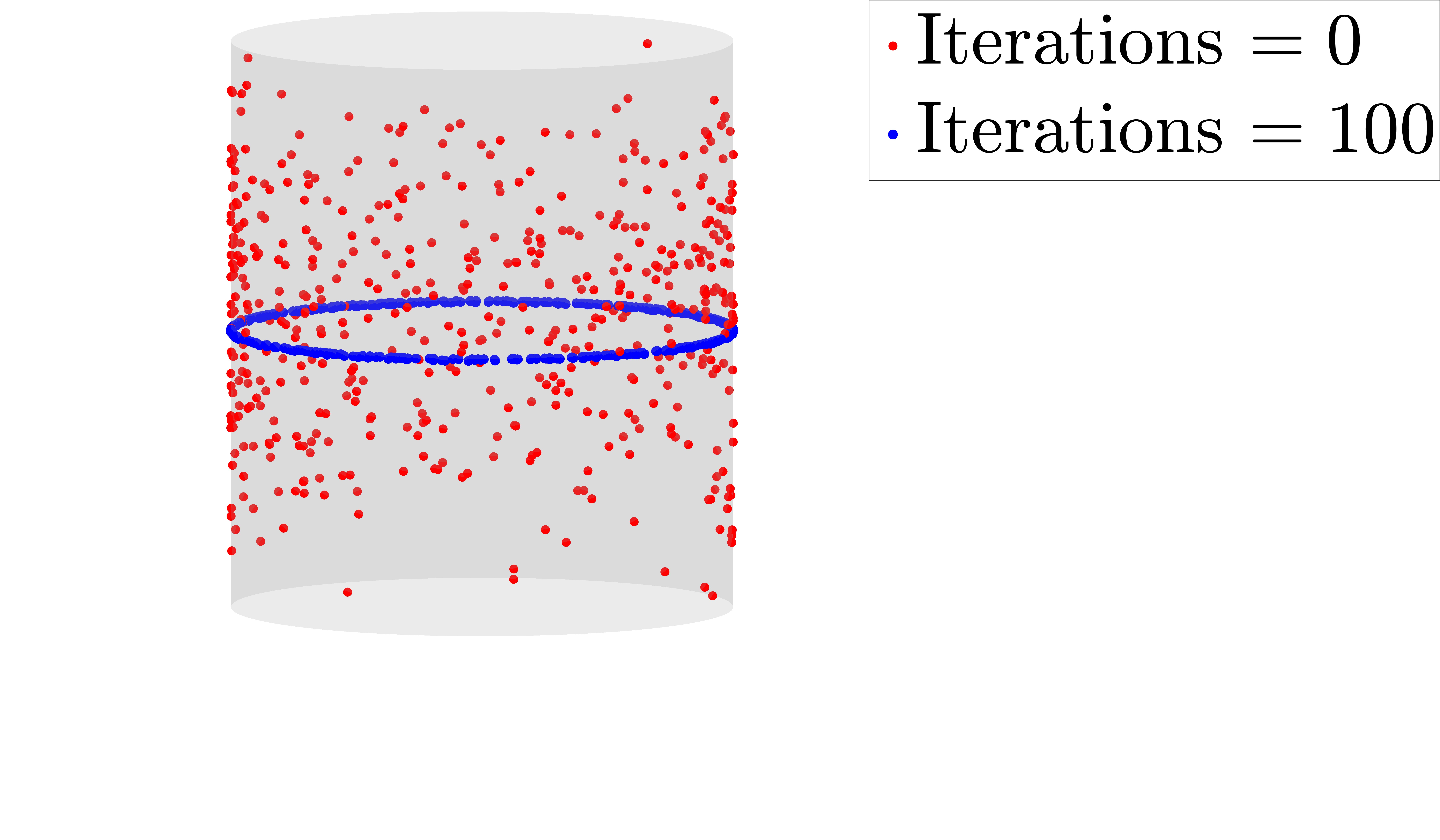}\\
                \caption    {The projections of the wrapping along $(2\pi,2\pi,2\pi)^{\top}$ on $\{y\in\mathbb{R}^4\mid y_{4}=0\}$}
                \label{sim-fig-tildepsi3-123}
            \end{minipage}
            \hfill
            \begin{minipage}[t]{0.45\textwidth}
                \centering
            \includegraphics[width=1.0\textwidth]{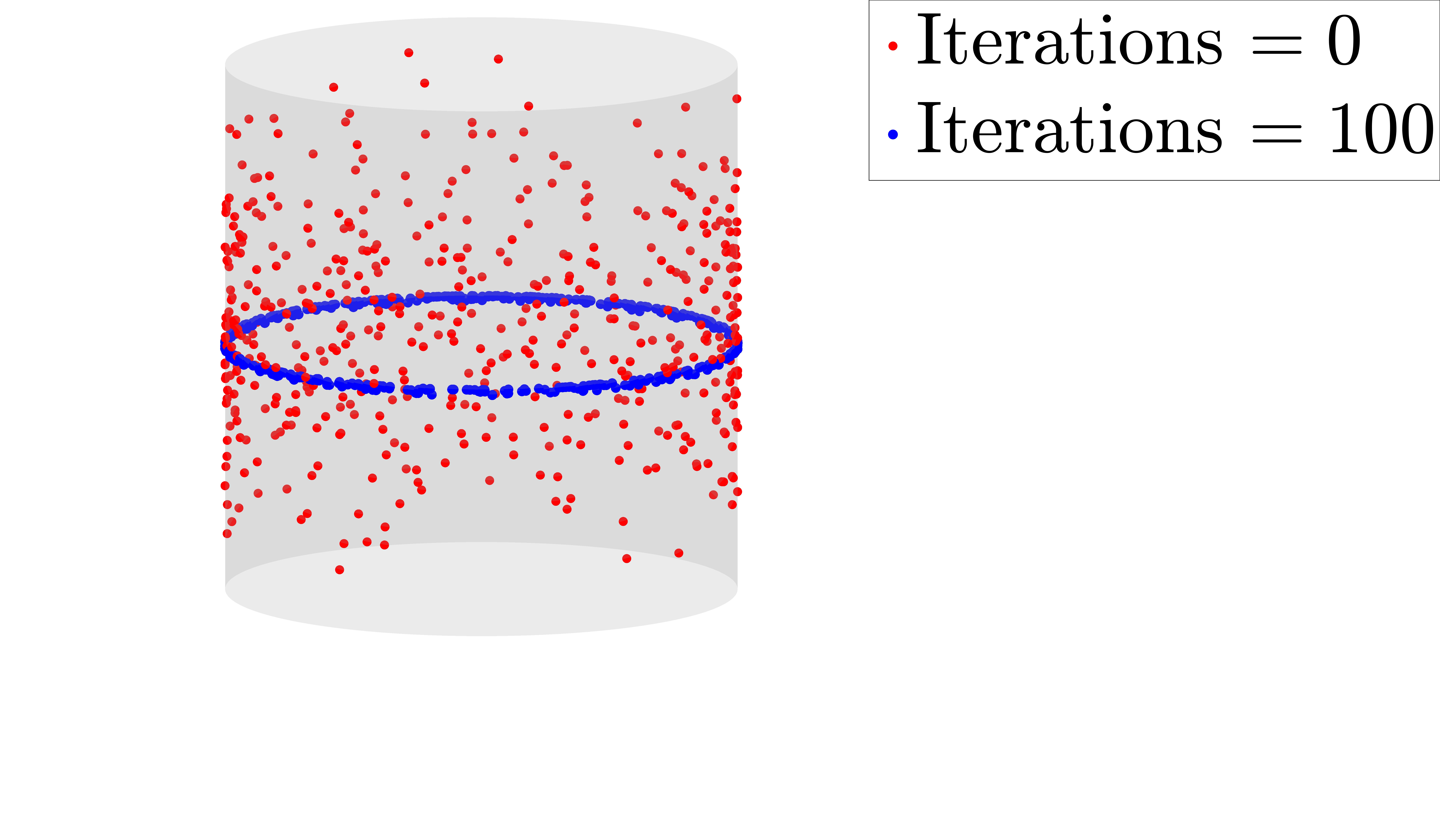}
                \caption {The projections of the wrapping along $(2\pi,2\pi,2\pi)^{\top}$  on  $\{y\in\mathbb{R}^4\mid y_{3}=0\}$}
                \label{sim-fig-tildepsi3-124}
            \end{minipage}
        \end{figure}



\end{document}